\documentclass[reqno]{amsart}
\usepackage{amsmath}
\usepackage{amssymb}
\usepackage{amsfonts}
\usepackage{amsthm}
\usepackage[foot]{amsaddr}
\usepackage{mathtools}
\mathtoolsset{%
}
\usepackage{mathscinet}

\usepackage{bbm}
\usepackage{bm}

\usepackage[utf8]{inputenc}
\usepackage[T1]{fontenc}

\usepackage[super]{nth}

\usepackage{libertine}
\usepackage[libertine,libaltvw,cmintegrals,varbb]{newtxmath}

\usepackage{dsfont}
\usepackage{mathrsfs}

\usepackage[%
cal=cm,
]
{mathalfa}

\usepackage[dvipsnames,svgnames]{xcolor}
\colorlet{MyBlue}{DodgerBlue!60!Black}
\colorlet{MyGreen}{DarkGreen!85!Black}

\usepackage[font=small,labelfont=bf]{caption}
\usepackage{subfigure}
\usepackage{tikz}
\usetikzlibrary{calc,patterns}

\usepackage{acronym}
\usepackage{booktabs}       % professional-quality tables
\usepackage{latexsym}
\usepackage{paralist}
\usepackage{wasysym}
\usepackage{xspace}

\usepackage[authoryear,compress,comma]{natbib}

\usepackage{hyperref}
\hypersetup{
colorlinks=true,
linktocpage=true,
pdfstartview=FitH,
breaklinks=true,
pdfpagemode=UseNone,
pageanchor=true,
pdfpagemode=UseOutlines,
plainpages=false,
bookmarksnumbered,
bookmarksopen=false,
bookmarksopenlevel=1,
hypertexnames=true,
pdfhighlight=/O,
urlcolor=MyBlue!60!black,linkcolor=MyBlue!70!black,citecolor=DarkGreen!70!black, % <--- for screen
pdftitle={},
pdfauthor={},
pdfsubject={},
pdfkeywords={},
pdfcreator={pdfLaTeX},
pdfproducer={LaTeX with hyperref}
}

\numberwithin{equation}{section}  %numberwithin goes before cleverefs when using hyperref
\usepackage[sort&compress,capitalize,nameinlink]{cleveref}
\crefname{app}{Appendix}{Appendices}

\crefrangeformat{equation}{\upshape(#3#1#4)\textendash(#5#2#6)}

\newcommand{\Z}{\mathbb{Z}}
\newcommand{\N}{\mathbb{N}}

\DeclareMathOperator*{\argmax}{arg\,max}
\DeclareMathOperator*{\argmin}{arg\,min}

\DeclareMathOperator{\ex}{\debug{\mathbb{E}}}

\DeclareMathOperator{\prob}{\debug{\mathbb{P}}}

\DeclareMathOperator{\val}{\debug{val}}

\DeclarePairedDelimiterX{\braket}[2]{\langle}{\rangle}{#1,#2}

\DeclarePairedDelimiterX{\inner}[2]{\langle}{\rangle}{#1,#2}
\DeclarePairedDelimiterX{\setdef}[2]{\{}{\}}{#1:#2}

\DeclarePairedDelimiterXPP{\probof}[1]{\prob}{(}{)}{}{%

#1}

\DeclarePairedDelimiterXPP{\exof}[1]{\ex}{[}{]}{}{%

#1}

\usepackage[textwidth=30mm]{todonotes}

\newcommand{\debug}[1]{#1}

\theoremstyle{plain}
\newtheorem{theorem}{Theorem}

\newtheorem*{corollary*}{Corollary}
\newtheorem{lemma}[theorem]{Lemma}
\newtheorem{proposition}[theorem]{Proposition}

\theoremstyle{definition}
\newtheorem{definition}[theorem]{Definition}
\newtheorem*{definition*}{Definition}

\newtheorem*{hypothesis*}{Hypothesis}

\theoremstyle{remark}
\newtheorem{remark}{Remark}
\newtheorem*{remark*}{Remark}
\newtheorem*{notation*}{Notational remark}
\newtheorem{example}{Example}

\numberwithin{theorem}{section}
\numberwithin{remark}{section}
\numberwithin{example}{section}

\newcommand{\x}{x_\lambda}
\newcommand{\y}{y_\lambda}
\newcommand{\vl}{v_\lambda}
\newcommand{\pj}{p_{-1}}
\newcommand{\pjj}{p_{0}}
\newcommand{\pjjj}{p_{+1}}
\newcommand{\qj}{q_{-1}}
\newcommand{\qjj}{q_{0}}
\newcommand{\qjjj}{q_{+1}}
\newcommand{\E}{\mathbb{E}}
\newcommand{\Proba}{\mathbb{P}}
\newcommand{\Scal}{\mathcal{S}}
\newcommand{\Tcal}{\mathcal{T}}
\newcommand{\1}{\mathbbm{1}}
\newcommand{\bsigma}{\bm\sigma}
\newcommand{\btau}{\bm\tau}
\newcommand{\bScal}{\bm\Scal}
\newcommand{\bTcal}{\bm\Tcal}
\newcommand{\bmu}{\bm\mu}
\newcommand{\bnu}{\bm\nu}

\begin{document}

%*************************************************************
%*****    FRONT MATTER AND METADATA
%*************************************************************

%----------------------------------------------------------------------
%%% TITLE & AUTHORS
%----------------------------------------------------------------------
\title
[Communicating Zero-Sum Product Stochastic Games]
{Communicating Zero-Sum Product Stochastic Games}

\author
[T.~Garrec]
{Tristan Garrec}
\address{TSE-R, Toulouse School of Economics,  Manufacture des Tabacs, 21 Allée de Brienne, 31015 Toulouse Cedex 6, France }
\email{tristan.garrec@ut-capitole.fr}
%\urladdr{\url{http://www.here.com}}

%\input{Thanks}

%----------------------------------------------------------------------
%%% KEYWORDS
%----------------------------------------------------------------------
%;Chinese-postman strategy

%----------------------------------------------------------------------
%%% ABSTRACT
%----------------------------------------------------------------------
\begin{abstract}

We study two classes of zero-sum stochastic games with compact action sets and a finite product state space. These two classes assume a communication property on the state spaces of the players. For strongly communicating on one side games, we prove the existence of the uniform value. For weakly communicating on both sides games, we prove that the asymptotic value, and therefore the uniform value, may fail to exist.

\end{abstract}
\acresetall

\maketitle

\section{Introduction}
\subsection{Problem and contribution}
In a zero-sum stochastic game, two players interact repeatedly at discrete times, with opposite interests. At each stage, players face a zero-sum game given by the state of nature which evolves according to the current state, and the pair of actions players choose given the history. Therefore, the actions played at each stage impact both the payoff today and the law of the state of nature tomorrow. Players intend to optimize their expected overall payoff. The $n$-stage repeated and the $\lambda$-discounted games are the games in which the overall payoffs are respectively the Cesàro and Abel means of the stage payoffs. Under mild assumptions both games have a value denoted respectively $v_n$ and $v_\lambda$.

A fundamental question arising in the theory of dynamic games is the asymptotic behavior of these values. We shall focus on the two following approaches of this issue. The asymptotic approach studies the convergence of the values of the $n$-stage repeated game and the $\lambda$-discounted game, as $n$ goes to infinity and $\lambda$ goes to $0$, that is as players become more patient. If these quantities converge and are equal, the game is said to have an asymptotic value. The uniform approach is dedicated to the existence, for both players, of strategies that are $\varepsilon$-optimal in every $n$-stage repeated game, provided that $n$ is large enough. If such strategies exist, players are able to play optimally in every game long enough without knowing the length of the game. In that case, the game is said to have a uniform value. These strategies are also $\varepsilon$-optimal in every $\lambda$-discounted game, provided that $\lambda$ is small enough. While the existence of the uniform value implies the existence of the asymptotic value, the converse is not true.

The aim of this paper is to study the asymptotic and uniform values in two classes of zero-sum product stochastic games. A product state space is of the form $X\times Y$. Moreover, players control the transitions on their own components of the state space, that is the next state in $X$ only depends on the current state in $X$ and the action of player 1, and similarly for $Y$. We consider the case where $X$ and $Y$ are finite, and action sets are compact. The two classes we are interested in assume a communication property on the state spaces of the players. These are called the strong and the weak communication properties --- the strong communication property implying the weak one.

The first class is the class of strongly communicating on one side zero-sum product stochastic games. In such games, for one player, there exists a time $T$ such that independently of his choice of policy, there is a positive probability of passing from any state to any other state in his component of the state space in exactly $T$ stages. This assumption, of ergodic nature, implies that the current state of the player having the strong communication property has in the long run little importance.

The second class is the class of weakly communicating on both sides zero-sum product stochastic games. In such games, for each player there exists a time $T$ and a policy such that, for any two states in their components of the state space, they can move from one to the other with positive probability in exactly $T$ stages. Thus players totally control the dynamics on their components of the state space.

Strongly communicating on one side and weakly communicating on both sides zero-sum product stochastic games have, to our knowledge, never been studied before.

For strongly communicating on one side games, we prove the existence of the uniform value, which does not depend on the initial state of the player having the strong communication property. Furthermore we prove that this player has $\varepsilon$-optimal strategies that have a simple structure. We call them Markov periodic strategies. Under these strategies, the action chosen at each stage does not depend on the whole history but only on the current state and stage modulo the period. Hence they are a particular case of Markov strategies, for which the actions chosen depend only on the current state and stage, and are more general than stationary strategies, which only depend on the current state. 
Our proof is based on a classification of the state space of the player who is not assumed to have the strong communication property (note that if both players have it, the proof is considerably simpler and the decomposition is actually not needed). This decomposition relies on recurrent classes induced by stationary policies. It has been introduced for Markov Decision Processes (MDP) by \cite{rossvaradarajan1991}, similar classifications have been used by \cite{bather1973,solan2003,fleschetal2008}. Building on that classification, we consider a family of auxiliary stochastic games and prove that they have a uniform value independent of the initial state. Finally we build an auxiliary MDP for the player who does not have the strong communication property, whose payoffs are the uniform values of the previous auxiliary games. We conclude by proving that the uniform value of the MDP is also the uniform value of the initial game.

Regarding weakly communicating on both sides games, we provide an example of a game which does not have an asymptotic value (and hence neither has a uniform one). Our proof is based on a reduction of this example to a simpler game of perfect information with two absorbing and two non absorbing states introduced by \cite{renault:hal-02130451}. A key ingredient of the non-existence of the asymptotic value is the non semi-algebraic aspect of the action set of one player. The example in particular shows that in weakly communicating on both sides games, even if players can go from any state to any other state of their component in finite time, they can make mistakes that are irreversible with regards to the joint state.

\subsection{Related literature}
Zero-sum stochastic games were introduced by \cite{shapley1953} in the finite setting (finite state and action sets), for which he proved the existence of the value in the $\lambda$-discounted game. \cite{mertensneyman1981} proved the existence of the uniform value. Their proof is based on the fact that the value of the $\lambda$-discounted game has bounded variations in $\lambda$, as shown by \cite{bewleykohlberg1976}. It is a key question whether the existence of the uniform value extends to non finite zero-sum stochastic games.

This question has been answered positively for several classes of zero-sum stochastic games with a finite state space and non finite action sets, as we consider in this paper. For absorbing games \citep{mertensetal2009} and recursive games \citep{lisorin2016}, the proofs use the operator approach of \cite{rosenbergsorin2001} that relies on the Shapley operator which entirely contains the dynamics of the game. Still for a finite state space, \cite{bolteetal2014} showed that games with semi-algebraic (or more generally definable) transitions and actions set have a uniform value. Finally, \cite{renault2010} proved the existence of the uniform value in MDPs with a finite state space and arbitrary action set.

However, in the last few years, several counterexamples to the existence of the asymptotic value in zero-sum stochastic games with finite state space and compact action sets have been proposed, see \citep{vigeral2013,sorinvigeral2015,ziliotto2016a}, ending the long standing idea that such games had an asymptotic value. \cite{larakirenault2017} provided such a counterexample with a product state space. As it is the case for the counterexample presented in this paper, all these counterexamples have in common to have non semi-algebraic transition probabilities or non semi-algebraic action sets. This is a key element to make the value oscillate. It should be put into perspective with the work of \cite{bolteetal2014} on definable zero-sum stochastic games. Another feature these counterexamples have in common is to have absorbing states. These states which cannot be left once reached are incompatible be the weak communication property. This is a major difference between previous counterexamples and the present one.

Zero-sum stochastic games on a product state space have been introduced by \cite{altmanetal2005}, who examined the case where each player only observes his component of the state and his actions, and showed that these games can be solved by linear programming. \cite{fleschetal2008,fleschetal2009} studied equilibria in $N$-players finite product stochastic games. They however considered an overall payoff which is the limit inferior of the $n$-stage repeated game payoff, while we are interested in the existence of the uniform value. Finally, \cite{larakirenault2017} showed the existence of the asymptotic value in zero-sum product stochastic games under a strong acyclicity condition. This strong acyclicity condition encompasses the irreversibility in the transitions of several classes of repeated games for which the asymptotic value is known to exist. The strong acyclicity condition is incompatible with our weak communication property. It is important to understand in which classes of stochastic games that do not assume an irreversibility condition the asymptotic value may exist.

\cite{gillette1957} introduced games in which independently of the choice of strategies, there is a positive probability of passing from any state to any other state in exactly $T$ stages. He called them cyclic stochastic games. These were also investigated by \cite{hoffmankarp1966}, \cite{bewleykohlberg1978} and \cite{vrieze2003}. A similar assumption has also been examined by \cite{fudenbergyamamoto2011} for games where players observe the state and a public signal related to the actions played. However in these articles the property is considered on the whole state space and not only on one component of a product state space.

\subsection{Organization of the paper}
The article is organized as follows. In \cref{sec_stochasticgames} we describe the model of zero-sum product stochastic games and recall some elementary facts. In \cref{sec_mainresults} we give formal definitions of strongly and weakly communicating, and state the two main theorems. Finally, \cref{sec_proofthstrong,sec_proofthweak} are dedicated to the proofs of the two main theorems.

\section{Preliminaries on zero-sum product stochastic games}
\label{sec_stochasticgames}
\subsection{Model and course of the game}

Let $X$ and $Y$ be two nonempty finite sets.
Let $A$ and $B$ be two nonempty compact metric sets endowed with their Borel $\sigma$-algebras.
Let $p : X\times A \to \Delta(X)$ and $q : Y\times B \to \Delta(Y)$, be such that for all $x,x'\in X$ and all $y,y'\in Y$, $p(x'|x,\cdot)$ and  $q(y'|y,\cdot)$ are continuous. $\Delta(X)$ denotes the set of probability measures over $X$, and similarly for $\Delta(Y)$.
Let $u : X\times Y\times A\times B \to [0,1]$, be such that for all $(x,y)\in X\times Y$ and all $a\in A$ and $b\in B$, $u(x,y,\cdot,b)$ and $u(x,y,a,\cdot)$ are continuous.

$X$ is the state space of player 1, $Y$ is the state space of player 2.
$A$ is the action set of player 1, $B$ is the action set of player 2. It is without loss of generality that the action sets do not depend on the current state.
$p$ is the transition probability of player 1, $q$ is the transition probability of player 2.
$u$ is the payoff to player 1.

Let $\Gamma = (X,Y,A,B,p,q,u).$ The game $\Gamma$ is played in stages as follows: an initial state $(x_1,y_1)\in X\times Y$ is given and known by the players. Inductively at stage $n$, knowing the past history $h_n = (x_1,y_1,a_1,b_1,\dots,x_{n-1},y_{n-1},a_{n-1},b_{n-1},x_n,y_n)$, player 1 and 2 simultaneously choose an action, respectively $a_n\in A$ and $b_n\in B$. A new state $x_{n+1}\in X$ is selected according to the distribution $p(\cdot|x_n,a_n)$ on $X$ and a new state $y_{n+1}\in Y$ is selected according to the distribution $q(\cdot|y_n,b_n)$ on $Y$. The payoff to player 1 at stage $n$ is $u_n = u(x_n,y_n,a_n,b_n)$. 

\subsection{Policies and Strategies}

Since the state space of the game $\Gamma$ is a product of two sets $X$ and $Y$, we distinguish policies of the players, which only depend on their own history, i.e., on their actions and component of the product state space, from strategies, which depend on the joint history. We denote strategies and strategy sets depending on the joint history with bold letters.

We denote the set of positive integers by $\N^\ast$. For $n\in \N^\ast$, let $H^1_n=X\times(A\times X)^{n-1}$ be the set of histories of player 1 at stage $n$ and $H^1_\infty = (A\times X)^\infty$ be the set of infinite histories. $H^1_n$ is endowed with the product $\sigma$-algebra $\mathcal{H}^1_n$, and $H^1_\infty$ with the product $\sigma$-algebra $\mathcal{H}^1_\infty$ spanned by $\bigcup_{n\geq 1} \mathcal{H}^1_n$. Let $\Scal$ denote the set of behavior policies of player 1 depending only on his own history. A policy $\sigma\in\Scal$ is a sequence $(\sigma_n)_{n\in\N^\ast}$, where $\sigma_n$ is a measurable map from $(H^1_n,\mathcal{H}^1_n)$ to $\Delta(A)$. A policy $\sigma$ together with an initial state $x\in X$ define a unique probability distribution over $H^1_\infty$ which we denote $\Proba_{\sigma}^{x}$. We define analogous objects for player 2 and denote $\Tcal$ the set of behavior policies of player 1 depending only on his own history.

For $n\in \N^\ast$, let $H_n=X\times Y\times(A\times B\times X\times Y )^{n-1}$ be the set of joint histories at stage $n$ and $H_\infty = (A\times B\times X\times Y)^\infty$ be the set of infinite joint histories. $H_n$ is endowed with the product $\sigma$-algebra $\mathcal{H}_n$, and $H_\infty$ with the product $\sigma$-algebra $\mathcal{H}_\infty$ spanned by $\bigcup_{n\geq 1} \mathcal{H}_n$. Let $\bScal$ and $\bTcal$ denote the sets of behavior strategies of player 1 and player 2 respectively. A strategy $\bsigma\in\bScal$ is a sequence $(\bsigma_n)_{n\in\N^\ast}$, where $\bsigma_n$ is a measurable map from $(H_n,\mathcal{H}_n)$ to $\Delta(A)$, and likewise for $\bTcal$. A pair of strategies $(\bsigma,\btau)$ together with an initial state $(x,y)$ define a unique probability distribution over $H_\infty$ which we denote $\Proba_{\bsigma,\btau}^{x,y}$.

A strategy is a Markov strategy if the mixed action played at every stage depends only on the current stage and state. Markov periodic strategies are Markov strategies depending on the stage modulo the period and on the current state. Let us give a formal definition.

\begin{definition}
For $N\in\N^\ast,$ a strategy $\bsigma$ of player 1 is called an $N$-periodic Markov strategy if there exists $(\bmu_n)_{n\in\{1,\dots,N\}}\in \Delta(A)^{X\times Y\times N}$ such that for all $n\in\N^\ast$ and all $h_n = (x_1,y_1,a_1,b_1,\dots,x_{n-1},y_{n-1},a_{n-1},b_{n-1},x_n,y_n)\in H_n$, $\bsigma_n(h_n) = \bmu_{n'}(x_n,y_n)$, where $n'\in\{1,\dots,N\}$ is equal to $n$ modulo $N$.
Markov $N$-periodic strategies of player 2 are defined likewise.
\end{definition}

Stationary strategies are Markov strategies depending only on the current state, hence they are Markov $1$-periodic strategies. Again, stationary strategies on the product state space $X\times Y$ (elements of $\Delta(A)^{X\times Y}$ and $\Delta(B)^{X\times Y}$ for player 1 and 2 respectively), are denoted with bold letters. Stationary policies of the players on their components of the state space (elements of $\Delta(A)^X$ and $\Delta(B)^Y$ respectively) are denoted with letters that are not bold.

Finally, a pair of stationary strategies $(\bmu,\bnu)\in \Delta(A)^{X\times Y}\times\Delta(B)^{X\times Y}$ induces a Markov chain $(X_n,Y_n)_{n\geq 1}$ over $X\times Y$. A state $(x',y')\in X\times Y$ is said to be accessible from $(x,y)$ (in $t$ stages) under $(\bmu,\bnu)$ if $\Proba^{x,y}_{\bmu,\bnu}((X_t,Y_t)=(x',y'))>0$. More generally, a property is said to hold under $(\bmu,\bnu)$ if it holds for the Markov chain induced on $X\times Y$. We use similar vocabulary for stationary policies on $X$ and on $Y$.

\subsection{\texorpdfstring{$N$}{N}-stage and \texorpdfstring{$\lambda$}{lambda}-discounted games}

For all $N\in \N^\ast $, the $N$-stage game $\Gamma_N$ starting in $(x,y)\in X\times Y$, is the game in which the payoff is \[\gamma_N(\bsigma,\btau)(x,y) = \frac{1}{N}\sum_{n=1}^N \E_{\bsigma,\btau}^{x,y} (u_n),\] for all $(\bsigma,\btau)\in \bScal\times \bTcal$.
The value of the $N$-stage game starting at $(x,y)$ is denoted $v_N(x,y)$.
It is characterized by the following recursive equation: 
\begin{align}
\label{eq_shapley_GammaN}
v_{N+1}(x,y) &= \val_{\mu\in\Delta(A),\nu\in\Delta(B)} \left[\frac{1}{N+1}u(x,y,\mu,\nu) + \frac{N}{N+1}\E_{\mu,\nu}^{x,y}(v_N)\right],
\end{align}
where 
\begin{equation*}
\begin{split}
&u(x,y,\mu,\nu) = \int_{A\times B} u(x,y,a,b)d\mu(a)d\mu(b),
\text{ and }\\
&\E_{\mu,\nu}^{x,y}(v_N) = \sum_{x',y'\in X\times Y} v_N(x',y')\int_{A\times B} p(x'|x,a)q(y'|y,b)d\mu(a)d\nu(b).
\end{split}
\end{equation*}

Moreover, by \citep[proposition 5.3]{Sor:Springer2002}, both players have optimal Markov strategies.

For all $\lambda\in (0,1]$, the $\lambda$-discounted game $\Gamma_\lambda$ starting in $(x,y)\in X\times Y$, is the game in which the payoff is \[\gamma_\lambda(\bsigma,\btau)(x,y) = \lambda\sum_{n=1}^{+\infty} (1-\lambda)^{n-1} \E_{\bsigma,\btau}^{x,y} (u_n),\] for all $(\bsigma,\btau)\in \bScal\times \bTcal$.
The value of the $\lambda$-discounted game starting at $(x,y)$ is denoted $v_\lambda(x,y)$.
It is characterized by the following fixed point equation:
\begin{align*}
v_\lambda(x,y) &= \val_{\mu\in\Delta(A),\nu\in\Delta(B)} \left[\lambda\cdot u(x,y,\mu,\nu) + (1-\lambda)\E_{\mu,\nu}^{x,y}(v_\lambda)\right].
\end{align*}

If $v_n$ and $v_\lambda$ converge as $n$ goes to infinity and $\lambda$ goes to $0$, and the limits are equal, then the game is said to have an asymptotic value.
The Tauberian theorem of \cite{ziliotto2016b} applies in this setting and the asymptotic value exists if $(v_\lambda)_{\lambda\in (0,1]}$ converges as $\lambda$ goes to $0$.

\subsection{Uniform value and optimal strategies}

Fix an initial state $(x,y)\in X\times Y$.
Player 1 is said to uniformly guarantee $v_\infty\in [0,1]$ if he has a strategy which guarantees $v_\infty$ (up to $\varepsilon$) against any strategy of player 2, in any game $\Gamma_N$, provided that $N$ is large enough. Formally,
\[\forall\varepsilon>0 \ \exists \bsigma\in\bScal \ \exists M\in\N^\ast \ \forall \btau\in \bTcal \ \forall N\geq M \ \gamma_N(\bsigma,\btau)(x,y) \geq v_\infty-\varepsilon.\]
And similarly for player 2. If both players guarantee $v_\infty$, then it is called the uniform value of the game $\Gamma$ starting at $(x,y)$.

Let $\varepsilon\geq 0$. A strategy $\bsigma\in \bScal$ is said to be (uniformly) $\varepsilon$-optimal for player 1 if 
\[\exists M\in\N^\ast \ \forall \btau\in \bTcal \ \forall N\geq M \ \gamma_N(\bsigma,\btau)(x,y) \geq v_\infty-\varepsilon.\] And similarly for player 2.

The next proposition states that if one of the players plays a stationary strategy, then the other player has an $\varepsilon$-optimal best response which is also stationary.

\begin{proposition}
\label{prop_bestreply}
Let $\bmu\in\Delta(A)^{X\times Y}$ be a stationary strategy of player 1 in the game $\Gamma$. Then, for all $\varepsilon>0$ there exist $\bnu\in\Delta(B)^{X\times Y}$ and $M\in \N^\ast$ such that for all $N\geq M$ and all $\btau\in\bTcal \ \gamma_N(\bmu,\bnu)(x,y)\leq\gamma_N(\bmu,\btau)(x,y)+\varepsilon.$
\end{proposition}

\begin{proof}
Let $\varepsilon>0$, $\bmu\in\Delta(A)^{X\times Y}$ and $(x,y)\in X\times Y$. When player 1 plays the stationary strategy $\bmu$, player 2 faces a Markov decision process having a uniform value $w_\infty$, and for which he has a uniformly $\varepsilon$-optimal stationary strategy $\bnu$, consult \citep[corollary 5.26]{Sor:Springer2002}. 

Thus, there exists $M\in\N^\ast$ such that for all $N\geq M$ one has $\gamma_N(\bmu,\bnu)(x,y)\leq w_\infty +\varepsilon$ and for all $\btau\in\bTcal$ one has $\gamma_N(\bmu,\btau)(x,y)\geq w_\infty-\varepsilon.$
\end{proof}

\section{Main results}
\label{sec_mainresults}
In the present article, we examine two communication properties on state spaces of the players. In words, a player has the strong communication property if there exists a time $T$ such that independently of his choice of policy (depending only on his own history), there is a positive probability of moving from any initial state of his component of the state space to any other state in exactly $T$ stages.

\begin{definition}
\label{def_stronglycommunicating}
Player 1 has the strong communication property if there exists $T\in\N^\ast$ such that for all policies $\sigma \in \Scal$ and all states $x,x'\in X$, one has $\Proba_{\sigma}^{x}(X_T = x') > 0$.
\end{definition}

A similar definition stands for player 2. A game is strongly communicating on one side if at least one player has the strong communication property.

In words, a player has the weak communication property if there exists a time $T$ and a policy (depending only on his own history) such that, for any initial state, he can reach any other state in his component of the state space in exactly $T$ stages, with positive probability.

\begin{definition}
\label{def_weak}
Player 1 has the weak communication property if there exists $T\in\N^\ast$ and a policy $\sigma \in \Scal$ such that for all states $x,x'\in X$, one has  $\Proba_{\sigma}^{x}(X_T = x') > 0$.
\end{definition}

A similar definition stands for player 2. A game is weakly communicating on both sides if both players have the weak communication property. Note that the strong communication property implies the weak one. 

The next simple example illustrates the notions of strong and weak communication.

\begin{example}
In \cref{fig_exemple_strweakcom} below we represent the state space and transition probabilities of player 1. The state space is $X = \{x,y,z\}$. Transition probabilities are represented by arrows between states. In state $x$ player 1 chooses $\alpha\in A$. With probability $\alpha$ the next state is $x$ and with probability $1-\alpha$ it is $y$. In state $y$ (resp. $z$) the next state is $y$ (resp. $x$) with probability $1$.

Hence if the action set $A$ equals $[0,1]$, player 1 has the weak communication property but does not have the strong communication property. Whereas if $A$ equals $[\varepsilon,1]$ for some $\varepsilon>0$, player 1 has the strong communication property.

\begin{figure}[ht]
\begin{center}
\begin{tikzpicture}[scale=1]

\node at (0,0){$z$};
\node at (1,1.732){$x$};
\node at (2,0){$y$};

\draw (0,0) circle (0.2);
\draw (1,1.732) circle (0.2);
\draw (2,0) circle (0.2);

\draw[->] (0,0.25) to [out=90,in=-135] (0.75,1.482);
\draw[->] (1.25,1.482) to [out=-45,in=90] (2,0.25);
\draw[->] (1.75,-0.25) to [out=-135,in=-45] (0.25,-0.25);

\draw[->] (0.8,1.9) to [out=135,in=180] (1,2.5)
	to [out=0,in=45] (1.2,1.9);		

\node [above] at (0,0.8) {$1$};
\node [above] at (2.3,0.8) {$1-\alpha$};
\node [below] at (1,-0.5) {$1$};
\node [above] at (1,2.5) {$\alpha$};

\end{tikzpicture}
\end{center}
\caption{Strong and weak communication}
\label{fig_exemple_strweakcom}
\end{figure}
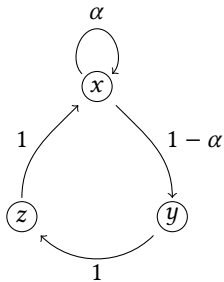
\end{example}

We prove in this article the following results.
\begin{theorem}
\label{THM_STRONG}
Any strongly communicating on one side zero-sum stochastic game has a uniform value.

Moreover, assuming player 1 has the strong communication property, the uniform value only depends on the initial state of player 2 and for all $\varepsilon>0$ player 1 has an $\varepsilon$-optimal Markov periodic strategy.
\end{theorem}

\begin{theorem}
\label{THM_WEAK}
There exists a weakly communicating on both sides zero-sum product stochastic game which does not admit an asymptotic value.
\end{theorem}

Remark that it is sufficient that one player has the strong communication property for the uniform value to exist, while both players having the weak communicating property does not ensure the existence of the asymptotic value.

The proof of \cref{THM_STRONG} relies on a classification of the state space of the player who is not assumed to have the strong communication property. We then consider a family of auxiliary stochastic games and prove that they have a uniform value independent of the initial state. Finally we build an auxiliary MDP for the player who is not assumed to have the strong communication property, whose payoffs are the uniform values of the previous auxiliary games. We conclude by proving that the uniform value of the MDP is also the uniform value of the initial game.

The proof of \cref{THM_WEAK} relies on a counterexample. We show that the Shapley equations of this game are also the Shapley equations of a simpler game with two absorbing and two non absorbing states, for which the asymptotic value does not exist.

\section{Proof of \texorpdfstring{Theorem~\ref{THM_STRONG}}{Theorem~\ref{THM_STRONG}}}
\label{sec_proofthstrong}

We assume in the proof of \cref{THM_STRONG} that player 1 has the strong communication property.

In \cref{subsec_classificationstates} we decompose the state space of player 2 in $L$ maximal communicating sets $C_i$ (\cref{def_stationaryrecurrentclass}). These sets, forming a partition of $Y$ (\cref{prop_partition}), allow us in \cref{subsec_auxiliarygamesrecurrentclass} to introduce $L$ auxiliary stochastic games $\Gamma_i$ having uniform values $v^i_\infty$ which do not depend on the initial state (\cref{prop_unif_Gamma_i}).

In \cref{subsec_auxiliaryMDP} we introduce a Markov Decision Process $\mathcal{G}$ for player 2 with the same transition probability $q$ as in the initial stochastic game, but with payoff $v^i_\infty$ when the state is in $C_i$. This MDP has a uniform value $w_\infty$, which is also the uniform value of the initial stochastic game $\Gamma$.

In \cref{subsec_P2unifvalue} we prove that player 2 uniformly guarantees $w_\infty$ in $\Gamma$ (\cref{prop_player2guarantees}). The idea is for player 2 to first play optimally in the MDP $\mathcal{G}$ (disregarding player 1) and then to switch to an optimal strategy in one of the games $\Gamma_i$.

Finally in \cref{subsec_P1unifvalue} we prove that player 1 uniformly guarantees $w_\infty$ (\cref{prop_player1guarantees}). This is one of the main difficulties of the proof. The natural idea is for player 1 to play optimally in each game $\Gamma_i$. However player 1 does not control the transitions from one $C_i$ to another (which may happen infinitely often). We deal with this issue by letting player 2 play a Markov periodic best response which prevents the state from jumping infinitely often between maximal communicating sets (\cref{lem_player1guarantee}). We conclude showing that the payoff this yields can also be obtained by player 2 as a limit payoff in $\mathcal{G}$ (\cref{prop_player1guarantees}).

\subsection{Classification of states}
\label{subsec_classificationstates}
The state space $Y$ is classified in a similar way to \citep{rossvaradarajan1991}, using recurrent classes induced by stationary policies on $Y$.

\begin{definition}
\label{def_stationaryrecurrentclass}
A subset $C$ of $Y$ is said to be a maximal communicating set if
\begin{itemize}
\item[i)] There exists a stationary policy on $Y$ such that $C$ is a recurrent class of the induced Markov chain on $Y$. Such a policy is said to be a stationary policy associated to $C$.
\item[ii)] $C$ is maximal, i.e., if there exists $C'$ a subset of $Y$ such that i) holds for $C'$ and $C\subseteq C'$, then $C'=C$.
\end{itemize}
\end{definition}

Let $C_1,\dots,C_L$ denote the maximal communicating sets.
$D$ denotes the set of transient states under every stationary policy (sometimes just called transient states for short).

\begin{remark}
Player 1 having the strong communication property, there is only one maximal communicating set in $X$, which is the whole state space $X$ itself.
\end{remark}

\cref{prop_partition} below can be found in \citep{rossvaradarajan1991} when the action space $B$ is finite, since the proof is similar with a compact action space we omit it. This result, showing that the maximal communicating sets and the set of states that are transient under any stationary policy form partition of $Y$, is fundamental in our proof of \cref{THM_STRONG} since it allows us (in the next section) to consider independent auxiliary games over each maximal communicating set.

\begin{proposition}
\label{prop_partition}
$\{C_1,\dots,C_L,D\}$ is a partition of $Y$.
\end{proposition}

The two examples given below allow to show how the state space of player 2 decomposes in maximal communicating sets, and what behavior the state process may have with regard to this decomposition. In particular, they show behaviors of the state process that may happen in the compact action setting and cannot happen in the finite one.

\begin{example}
\label{example_scc1}
In \cref{fig_exemple_classification} the state space of player 2 is $Y = \{x,y,z\}$ and the action space is $B = [0,1/2]$. Again, the transitions probabilities are represented on the arrows between states, e.g., playing $\beta\in B$ in state $y$, the next state is $x$ with probability $\beta$, $y$ with probability $1-\beta-\beta^2$ and $z$ with probability $\beta^2$. The maximal communicating sets are $\{y\}$ and $\{z\}$ and the set of transient states under every stationary policy is $\{x\}$.

\begin{figure}[ht]
\begin{center}
\begin{tikzpicture}[scale=1]
\node at (-3,0){$x$};
\node at (0,0){$y$};
\node at (3,0){$z$};

\draw (0,0) circle (0.2);
\draw (-3,0) circle (0.2);
\draw (3,0) circle (0.2);

\draw[->] (-2.8,-0.2) to [out=-45 ,in=-135] (-0.2,-0.2);
\draw[->] (-0.2, 0.2) to [out= 135 ,in= 45] ( -2.8, 0.2);
\draw[->] (0.3, 0) to [out= 0 ,in= 180] ( 2.7, 0);

\draw[->] (3.2,0.2) to [out=45,in=90] (4,0)
	to [out=-90,in=-45] (3.2,-0.2);
\draw[->] (-0.2,0.2) to [out=135,in=180] (0,0.8)
	to [out=0,in=45] (0.2,0.2);		

\node [below] at (-1.5,-0.8) {$1$};
\node [above] at (-1.5,0.8) {$\beta$};
\node [above] at (0,0.8) {$1-\beta-\beta^2$};
\node [above] at (1.5,0) {$\beta^2$};
\node [right] at (4,0) {$1$};

\end{tikzpicture}
\end{center}
\caption{States, actions and transitions of player 2}
\label{fig_exemple_classification}
\end{figure}
Remark that, if the initial state is $x$, playing $\beta = \frac{1}{2n}$ in state $y$ at stage $n\geq 1$, player 2 has a positive probability of switching infinitely often between the transient state $x$ and the maximal communicating set $\{y\}$.
This cannot happen when the action space is finite. In that case, for any policy of player 2, after finitely many stages the process $(Y_n)_{n\geq1}$ remains forever in one of the maximal communicating sets with probability $1$, see \citep[lemma 2 and proposition 2]{rossvaradarajan1991}.
\end{example}

The next example is derived from the previous one and shows another possible behavior of the state process in the compact action setting.

\begin{example}
Here, see \cref{fig_exemple_classification2}, the state space of player 2 is $Y = \{x,y,z\}$ and the action space is $B = [0,1/2]$. The maximal communicating sets are $\{x\}$, $\{y\}$ and $\{z\}$.

\begin{figure}[ht]
\begin{center}
\begin{tikzpicture}[scale=1]
\node at (-3,0){$x$};
\node at (0,0){$y$};
\node at (3,0){$z$};

\draw (0,0) circle (0.2);
\draw (-3,0) circle (0.2);
\draw (3,0) circle (0.2);

\draw[->] (-2.8,-0.2) to [out=-45 ,in=-135] (-0.2,-0.2);
\draw[->] (-0.2, 0.2) to [out= 135 ,in= 45] ( -2.8, 0.2);
\draw[->] (0.3, 0) to [out= 0 ,in= 180] ( 2.7, 0);

\draw[->] (-3.2,0.2) to [out=135,in=90] (-4,0)
	to [out=-90,in=-135] (-3.2,-0.2);
\draw[->] (3.2,0.2) to [out=45,in=90] (4,0)
	to [out=-90,in=-45] (3.2,-0.2);
\draw[->] (-0.2,0.2) to [out=135,in=180] (0,0.8)
	to [out=0,in=45] (0.2,0.2);		

\node [left] at (-4,0) {$1-2\alpha$};
\node [below] at (-1.5,-0.8) {$2\alpha$};
\node [above] at (-1.5,0.8) {$\beta$};
\node [above] at (0,0.8) {$1-\beta-\beta^2$};
\node [above] at (1.5,0) {$\beta^2$};
\node [right] at (4,0) {$1$};

\end{tikzpicture}
\end{center}
\caption{States, actions and transitions of player 2}
\label{fig_exemple_classification2}
\end{figure}
Again, remark that if the initial state is $x$, playing $\alpha = 1/2$ in state $x$ and $\beta = \frac{1}{2n}$ in state $y$ at stage $n\geq 1$, player 2 has a positive probability of switching infinitely often between maximal communicating sets $\{x\}$ and $\{y\}$.
\end{example}
As we mentioned earlier, this behavior the state of player 2 may have gives rise to a difficulty since player 1 cannot control the jumps of player's 2 state between maximal communicating sets or transient states. This difficulty is defused by letting the players play Markov periodic strategies. 

Another way of defining maximal communicating sets is via pairs of stationary strategies of player 1 and 2 on $X\times Y$, as in the following definition. Recall that a pair of stationary strategies $(\bmu,\bnu)\in \Delta(A)^{X\times Y}\times\Delta(B)^{X\times Y}$ induces a Markov chain over $X\times Y$, which may have recurrent classes or transient states.

\begin{definition}
A subset $C$ of $X\times Y$ is said to be a joint maximal communicating set if
\begin{itemize}
\item[i)] There exists a pair of stationary strategies in $\Delta(A)^{X\times Y}\times \Delta(B)^{X\times Y}$ such that $C$ is a recurrent class of the induced Markov chain on $X\times Y$.
\item[ii)] $C$ is maximal, i.e., if there exists $C'$ a subset of $X\times Y$ such that i) holds for $C'$ and $C\subseteq C'$, then $C'=C$. 
\end{itemize}
\end{definition}

Advantageously, maximal communicating sets and joint maximal communicating sets match in following sense.

\begin{proposition}
\label{prop_jointscc}
The joint maximal communicating sets are $X\times C_1,\dots,X \times C_L$, and the set of states that are transient under any pair $(\bmu,\bnu)$ of stationary strategies is $X\times D$.
\end{proposition}

\begin{proof}
Let $i\in \{1,\dots,L\}$. $X\times C_i$ is a recurrent class in $X\times Y$ for some pair of stationary policies of player 1 on $X$ and player 2 on $Y$ (just take the stationary policy of player 2 on $Y$ associated to $C_i$ and any stationary policy of player 1 on $X$).

Let us show that $X\times C_i$ is maximal. Suppose that there exists a subset $C$ of $X\times Y$ such that $C$ is a recurrent class under a pair of stationary strategies $(\bmu,\bnu)\in \Delta(A)^{X\times Y}\times\Delta(B)^{X\times Y}$ and $X\times C_i \subseteq C$.

Let $C' = \{y\in Y \ | \ \exists x\in X \ (x,y)\in C\}$ be the projection of $C$ over $Y$.
For all $y\in C'$, define $C_y = \{x\in X \ | \ (x,y)\in C\}$.
Finally, define $\nu\in\Delta(B)^Y$ by, for all $y\in C'$
\[\nu(y)=\frac{1}{|C_y|}\sum_{x\in C_y} \bnu(x,y),\]
and arbitrarily outside $C'$.

$C'$ is closed under $\nu$ and every state in $C'$ is accessible from any other state. Hence $C'$ is a recurrent class. Thus $C' \subseteq C_i$, and $C = X\times C_i$.

Clearly $X\times D = X\times Y \setminus \left(\bigcup\limits_{i=1}^L X\times C_i\right)$ is the set of transient states.
\end{proof}

\subsection{The auxiliary games over each maximal communicating set}
\label{subsec_auxiliarygamesrecurrentclass}

Consider now a family of $L$ zero-sum product stochastic games $(\Gamma_i)_{i\in\{1,\dots,L\}}$. For all $i \in \{1,\dots,L\}$, if $y\in C_i$, define the set of actions of player 2 at state $y$ under which the state has probability $1$ of staying in $C_i$, \[B_y = \{b\in B \ | \ q(C_i|y,b)=1\}.\]
We remark that the set $B_y$ is closed.
Then the game $\Gamma_i$ is given by \[\Gamma_i = (X,C_i,A,(B_y)_{y\in C_i},p,q,u).\]
The sets of strategies for player 1 and player 2 in $\Gamma_i$ are respectively denoted $\bScal_i$ and $\bTcal_i$.
The value of the $N$-stage game starting at $(x,y)\in X\times C_i$, is denoted $v_N^i(x,y)$.

In \cref{lem_constantvalue} and \cref{prop_unif_Gamma_i} below we prove that each of these games have a uniform value that does not depend on the initial state of the game in $X\times C_i$.

\begin{lemma}
\label{lem_constantvalue}
Let $v^i:X\times C_i \to [0,1]$ be any uniform limit point of the sequence $(v_N^i)_{N\geq 1}$. Then $v^i$ is constant over $X\times C_i$.
\end{lemma}

\begin{proof}
Let $i\in \{1,\dots,L\}$ and $(x_1^\ast,y_1^\ast)\in \argmax_{X\times C_i} v^i(\cdot)$. Since player 1 has the strong communication property
\begin{align}
\label{ali_strongcom}
\forall t\geq T \ \forall x'\in X \ \forall\sigma\in\Scal \ \Proba_\sigma^{x_1^\ast}(X_T=x')>0,
\end{align}
and since $C_i$ is a maximal communicating set
\begin{align}
\label{ali_scc}
\exists \nu\in \Delta(B)^Y \ \forall y'\in C_i \ \exists t\geq T \ \Proba_{\nu}^{y_1^\ast}(Y_t=y')>0.
\end{align}
We now fix $(x',y')\in X\times C_i$ and let $\nu\in\Delta(B)^Y$ and $t\geq T$ be as in statement (\ref{ali_scc}). We denote $(x_t^\ast,y_t^\ast)=(x',y')$. Then there exist $y_2^\ast,...,y_{t-1}^\ast\in C_i$ such that \[q(y_t^\ast|y_{t-1}^\ast,\nu(y_{t-1}^\ast))\dots q(y_2^\ast|y_1^\ast,\nu(y_1^\ast))>0.\]
Passing to the limit (uniformly) in the Shapley equation (\cref{eq_shapley_GammaN}) one has
\begin{align*}
\label{eq_shapleylimit}
v^i(x_1^\ast,y_1^\ast) &= \max_{\mu_1\in\Delta(A)}\min_{\nu_1\in\Delta(B_{y_1^\ast})} \sum\limits_{\substack{x_2\in X\\ y_2\in C_i}} p(x_2|x_1^\ast,\mu_1) q(y_2|y_1^\ast,\nu_1) v^i(x_2,y_2).
\end{align*}
Hence
\begin{align*}
v^i(x_1^\ast,y_1^\ast) &\leq \max_{\mu_1\in\Delta(A)}\sum\limits_{\substack{x_2\in X\\ y_2\in C_i}} p(x_2|x_1^\ast,\mu_1) q(y_2|y_1^\ast,\nu(y_1^\ast)) v^i(x_2,y_2) \leq v^i(x_1^\ast,y_1^\ast).
\end{align*}
Thus there exists $\mu_1^\ast\in\Delta(A)$ such that
\begin{align*}
v^i(x_1^\ast,y_1^\ast) &= \sum\limits_{\substack{x_2\in X\\ y_2\in C_i}} p(x_2|x_1^\ast,\mu_1^\ast) q(y_2|y_1^\ast,\nu(y_1^\ast))v^i(x_2,y_2).
\end{align*}
Finally there exists $\mu_1^\ast\in\Delta(A)$ such that
\begin{align*}
v^i(x_1^\ast,y_1^\ast) &= \sum\limits_{x_2\in X} p(x_2|x_1^\ast,\mu_1^\ast)v^i(x_2,y_2^\ast).
\end{align*}
Let $k\in\{2,\dots,t-1\}$ and suppose there exist $\mu_1^\ast\in\Delta(A)$ and $\mu_2^\ast,\dots,\mu_{k-1}^\ast\in\Delta(A)^X$ such that
\begin{align*}
v^i(x_1^\ast,y_1^\ast) &= \sum\limits_{x_2,\dots,x_{k-1}\in X}p(x_{k-1}|x_{k-2},\mu_{k-2}^\ast)\dots p(x_2|x_1^\ast,\mu_1^\ast) v^i(x_{k-1},y_{k-1}^\ast).
\end{align*}
Then there exist $\mu_1^\ast\in\Delta(A)$ and $\mu_2^\ast,\dots,\mu_{k-1}^\ast\in\Delta(A)^X$ such that
\begin{align*}
v^i(x_1^\ast,y_1^\ast) &\leq \sum\limits_{x_2,\dots,x_{k-1}\in X} p(x_{k-1}|x_{k-2},\mu_{k-2}^\ast)\dots p(x_2|x_1^\ast,\mu_1^\ast)\\
&\max_{\mu_{k-1}\in\Delta(A)}\min_{\nu_{k-1}\in\Delta(B_{y_{k-1}^\ast})} \sum\limits_{\substack{x_k\in X\\ y_k\in C_i}} p(x_k|x_{k-1},\mu_{k-1}^\ast) q(y_k|y_{k-1}^\ast,\nu_{k-1})v^i(x_k,y_k)\\
&\leq \sum\limits_{x_2,\dots,x_{k-1}\in X} p(x_{k-1}|x_{k-2},\mu_{k-2}^\ast)\dots p(x_2|x_1^\ast,\mu_1^\ast)\\
&\max_{\mu_{k-1}\in\Delta(A)} \sum\limits_{\substack{x_k\in X\\ y_k\in C_i}} p(x_k|x_{k-1},\mu_{k-1}^\ast) q(y_k|y_{k-1}^\ast,\nu(y_{k-1}^\ast))v^i(x_k,y_k)\\
&\leq v^i(x_1^\ast,y_1^\ast).
\end{align*}
Thus there exists $\mu_k^\ast\in\Delta(A)^X$ such that
\begin{align*}
v^i(x_1^\ast,y_1^\ast) &= \sum\limits_{x_2,\dots,x_{k-1}\in X} p(x_{k-1}|x_{k-2},\mu_{k-2}^\ast)\dots p(x_2|x_1^\ast,\mu_1^\ast)\\&
\sum\limits_{\substack{x_k\in X\\ y_k\in C_i}} p(x_k|x_{k-1},\mu_{k-1}^\ast) q(y_k|y_{k-1}^\ast,\nu(y_{k-1}^\ast))v^i(x_k,y_k).
\end{align*}
Finally,
\begin{align*}
v^i(x_1^\ast,y_1^\ast) = \sum\limits_{x_2,\dots,x_{k-1},x_k\in X} p(x_k|x_{k-1},\mu_{k-1}^\ast) &p(x_{k-1}|x_{k-2},\mu_{k-2}^\ast)\dots\\& p(x_2|x_1^\ast,\mu_1^\ast) v^i(x_k,y_k^\ast).
\end{align*}
Thus by induction, there exist $\mu_1^\ast\in\Delta(A)$ and $\mu_2^\ast,\dots,\mu_{t-1}^\ast\in\Delta(A)^X$ such that for all $x_2,\dots,x_{t-1}\in X$ \[0 = p(x_t^\ast|x_{t-1},\mu_{k-1}^\ast (x_{t-1})) \dots p(x_3|x_2,\mu_2^\ast(x_2))p(x_2|x_1^\ast,\mu_1^\ast) (v^i(x_t^\ast,y_t^\ast)-v^i(x_1^\ast,y_1^\ast)).\]
Let us define the policy $\sigma\in\Scal$ by $\sigma_1(x_1^\ast) = \mu_1^\ast$ and for all $k\in \{2,\dots,t-1\}$ $\sigma_k(h^1_k) = \mu^\ast_k(x_k)$. Finally $\sigma$ is defined arbitrarily for all $k\geq t$.
We have from statement (\ref{ali_strongcom}) that there exists $x_2^\ast,\dots,x_{t-1}^\ast\in X$ such that 
\[p(x_t^\ast|x_{t-1}^\ast,\mu_{t-1}^\ast)\dots p(x_2^\ast|x_1^\ast,\mu_1^\ast)>0.\]
Finally, one has $v^i(x_t^\ast,y_t^\ast) = v^i(x_1^\ast,y_1^\ast).$
\end{proof}

Since any uniform limit point of $(v_N^i)_{N\geq 1}$ is constant over $X\times C_i$, we show that players guarantee $\limsup_{N\to+\infty} v^i_N$ by playing an $\varepsilon$-optimal Markov $N_0$-periodic strategy for a sufficiently large $N_0$. 

\begin{proposition}
\label{prop_unif_Gamma_i}
For all $i\in\{1,\dots,L\}$ the game $\Gamma_i$ has a uniform value $v_\infty^i$, which is constant over $X\times C_i$. 

Moreover, for all $\varepsilon>0$ there exists $N_0\in\N^\ast$ such that both players have an $\varepsilon$-optimal Markov $N_0$-periodic strategy in each $\Gamma_i$.
\end{proposition}

\begin{proof}
Here we take $v^i = \limsup_{N\to+\infty} v^i_N$. Let $\varepsilon>0$. Thanks to \cref{lem_constantvalue}, there exists $N^i_0\geq 1$ such that for all $(x,y)\in X\times C_i$ one has $|v^i_{N^i_0}(x,y)-v^i|\leq\varepsilon$.
Hence \[\max_{\bsigma\in\bScal_i}\min_{\btau\in\bTcal_i}\E_{\bsigma,\btau}^{x,y}\left[\sum_{n=1}^{N^i_0}u_n\right] \geq N^i_0(v^i-\varepsilon).\]
Recall that there exists a Markov strategy $\bsigma^i = (\bsigma^i_n)_{n\in[1,N^i_0]}$, independent of the initial state $(x,y)$, such that for all $\tau\in\Tcal_i$ \[\E_{\bsigma^i,\btau}^{x,y}\left[\sum_{n=1}^{N^i_0}u_n\right] \geq N^i_0(v^i-\varepsilon).\]
Let $N=p N^i_0 + r$, $p\geq 1$ and $r\leq N^i_0-1$ be two integers. The following Markov $N_0$-periodic strategy is still denoted $\bsigma^i$ : at stage $n\geq 1$, in state $(x,y)\in X\times C_i$, play $\sigma^i_{n'}(x,y)$, where $n'$ equals $n$ modulo $N_0^i$. Then for all $\btau\in\Tcal_i$
\begin{align*}
\E_{\bsigma^i,\btau}^{x,y}\left[\sum_{n=1}^{N} u_n\right] &= \sum_{k=0}^{p-1}\E_{\bsigma^i,\btau}^{x,y}\left[\sum_{n=kN^i_0+1}^{(k+1)N^i_0}u_n\right] + \E_{\bsigma^i,\btau}^{x,y}\left[\sum_{n=pN^i_0+1}^{N}u_n\right]\\
&\geq pN^i_0(v^i-\varepsilon).
\end{align*}
Finally, for all $N\geq \frac{N^i_0}{\varepsilon}$,
\[v^i_N(x,y)\geq \frac{pN^i_0(v^i-\varepsilon)}{N}\geq (1-\varepsilon)(v^i-\varepsilon).\]
That is, player 1 uniformly guarantees $v^i$ and $\bsigma^i$ is an $\varepsilon$-optimal Markov $N_0^i$-periodic strategy. A similar proof shows that player 2 also uniformly guarantees $v^i$ and has an $\varepsilon$-optimal Markov $N_0^i$-periodic strategy $\tau^i$.

Note that $N_0^i$ can be taken uniformly over $\{1,\dots,L\}$ by setting $N_0=\prod_{i=1}^L N_0^i.$
\end{proof}

\subsection{The auxiliary Markov decision process \texorpdfstring{$\mathcal{G}$}{G}}
\label{subsec_auxiliaryMDP}

Consider the Markov decision process $\mathcal{G} = (Y,B,q,g),$ in which player 2 is the only decision maker and his aim is to minimize
\[\begin{array}{ccccl}
g & : & Y & \to & [0,1] \\
 & & y & \mapsto & \begin{cases} v_\infty^i \text{ if there exists } i\in\{1,\dots,L\} \text{ such that } y\in C_i\\ 
1/2 \text{ if } y\in D.\end{cases}\\
\end{array}\]

Recall, see \citep{Sor:Springer2002}, that, $\mathcal{G}$ has a uniform value $w_\infty\in [0,1]^Y$ and that for every $\varepsilon>0$, player 2 has an $\varepsilon$-optimal stationary policy.

The interpretation is the following. The objective of player 2 is to reach the maximal communicating set $C_i$ with corresponding auxiliary game $\Gamma_i$ having the lowest uniform value $v_\infty^i$ possible, and stay in $C_i$. Note that the payoff of $1/2$ in $D$ is arbitrary and does not change the value of $w_\infty$.

We will prove that the uniform value $w_\infty$ of the MDP $\mathcal{G}$ is in fact also the uniform value of the initial game $\Gamma$.

\subsection{Player 2 uniformly guarantees \texorpdfstring{$w_\infty$}{w} in \texorpdfstring{$\Gamma$}{Gamma}}
\label{subsec_P2unifvalue}

\begin{proposition}
\label{prop_player2guarantees}
Player 2 uniformly guarantees $w_\infty$ in $\Gamma$.
\end{proposition}

To prove \cref{prop_player2guarantees}, we show that player 2 has an $\varepsilon$-optimal strategy which has a rather simple structure: play according to some $\varepsilon$-optimal stationary policy in $\mathcal{G}$ on $Y$, until reaching a recurrent class of the induced Markov chain on $Y$. Then switch to an $\varepsilon$-optimal Markov periodic strategy in some game $\Gamma_i$ having the property that the state remains in the corresponding maximal communicating set $C_i$ (indeed the $\varepsilon$-optimal strategies used in the proof of \cref{prop_player2guarantees} below can be taken Markov periodic thanks to \cref{prop_unif_Gamma_i}).

\begin{proof}
Let $\varepsilon>0$, and let $\nu_\mathcal{G}$ be an $\varepsilon$-optimal stationary policy of player 2 in $\mathcal{G}$.
$\nu_\mathcal{G}$ induces $l$ recurrent classes $R_1,\dots,R_l$ over $Y$. Moreover, by \cref{def_stationaryrecurrentclass}, there exists a mapping $\varphi : \{1,\dots,l\}\to\{1,\dots,l\}$ such that for all $i\in\{1,\dots,l\}$, $R_i\subseteq C_{\varphi(i)}$.

For all $i\in\{1,\dots,L\}$ define $T_i= \min\{n\in\N^\ast \ | \ Y_n\in R_i\}$ to be the hitting time of $R_i$ by $(Y_n)_{n\in\N^\ast}$. The minimum over an empty set is taken equal to $+\infty$.

We now consider the game $\Gamma$ with initial state $(x,y)\in X\times Y$. Let us define the strategy $\overline{\btau}\in\bTcal$ of player 2 as follows. Until there exists $i\in\{1,\dots,l\}$ such that the state of player 2 is in $R_i$ play $\nu_\mathcal{G}$. Let $n\in \N^\ast$ be the first stage at which the state of player 2 reaches one of the recurrent classes $R_i$, $i\in\{1,\dots,l\}$. From stage $n$ on, play $\btau^{\varphi(i)}$ which is an $\varepsilon$-optimal strategy in $\Gamma_{\varphi(i)}$.

Let $\bsigma$ be any strategy of player 1 in $\Gamma$. Remark that under $(\bsigma,\overline{\btau}),$ the laws of the $T_i$'s are the same as under $\nu_\mathcal{G}$.

Let $N\geq 2$,
\begin{align}
\frac{1}{N}\E_{\bsigma,\overline{\btau}}^{x,y}\left[ \sum_{n=1}^N u_n\right] &=\frac{1}{N}\sum_{i=1}^l\E_{\bsigma,\overline{\btau}}^{x,y}\left[ \left(\sum_{n=1}^{T_i}u_n \right)\1\left(\sqrt{N}>T_i = \min_{k\in\{1,\dots,l\}} T_k\right)\right]\label{eq_p2_1} \\ 
+&\frac{1}{N}\sum_{i=1}^l\E_{\bsigma,\overline{\btau}}^{x,y}\left[ \left(\sum_{n=T_i+1}^{N}u_n \right)\1\left(\sqrt{N}>T_i = \min_{k\in\{1,\dots,l\}} T_k\right)\right]\label{eq_p2_2} \\ 
+&\frac{1}{N}\sum_{i=1}^l\E_{\bsigma,\overline{\btau}}^{x,y}\left[ \left(\sum_{n=1}^{N}u_n \right)\1\left(\sqrt{N}\leq T_i = \min_{k\in\{1,\dots,l\}} T_k\right)\label{eq_p2_3} \right].
\end{align}
Since the payoffs $u_n$ are at most $1$ and $T_i$ is at most $\sqrt{N}$ in the indicator function, (\ref{eq_p2_1}) is at most
\begin{equation*}
\frac{\sqrt{N}}{N}\sum_{i=1}^l \Proba_{\bsigma,\overline{\btau}}^{x,y}\left(\sqrt{N}>T_i = \min_{k\in\{1,\dots,l\}} T_k\right),
\end{equation*}
which itself is at most $l\frac{\sqrt{N}}{N}$ which is smaller than $\varepsilon$ for $N$ large enough.

Remark that in (\ref{eq_p2_2}), $\frac{1}{N}\leq \frac{1}{N-T_i}$, hence this quantity is less than or equal to 
\begin{equation*}
\sum_{i=1}^l \E_{\bsigma,\overline{\btau}}^{x,y}\left[\frac{1}{N-T_i} \left(\sum_{n=T_i+1}^N u_n\right)\1\left(\sqrt{N}>T_i = \min_{k\in\{1,\dots,l\}} T_k\right)\right],
\end{equation*}
the latter equals
\begin{align*}
\sum_{i=1}^l \E_{\bsigma,\overline{\btau}}^{x,y}\left[\frac{1}{N-T_i} \sum_{n=T_i+1}^N u_n\Bigg\vert\sqrt{N}> T_i = \min_{k\in\{1,\dots,l\}} T_k\right]\Proba_{\bsigma,\overline{\btau}}^{x,y}\left(\sqrt{N}>T_i = \min_{k\in\{1,\dots,l\}} T_k\right).
\end{align*}
Recall that for all $i\in\{1,\dots,L\}$, $\btau^i$ is $\varepsilon$-optimal in $\Gamma_i$ which has uniform value $v^i_\infty$. For $N$ large enough the quantity above is less than or equal to
\begin{align*}
\sum_{i=1}^l \left(v^{\varphi(i)}_\infty+\varepsilon\right)\Proba_{\nu_\mathcal{G}}^{y}\left( T_i = \min_{k\in\{1,\dots,l\}} T_k\right),
\end{align*}
because under $(\bsigma,\overline{\btau})$, the laws of the $T_i$'s are the same as under $\nu_\mathcal{G}$.
And since $\nu_\mathcal{G}$ is $\varepsilon$-optimal in $\mathcal{G}$, $\sum_{i=1}^l v^{\varphi(i)}_\infty\Proba_{\nu_\mathcal{G}}^{y}\left( T_i = \min_{k\in\{1,\dots,l\}} T_k\right)$ is less than or equal to $w_\infty(y)+\varepsilon$.

Finally, since the payoffs $u_n$ are at most $1$, (\ref{eq_p2_3}) is less than or equal to
\begin{equation*}
\sum_{i=1}^l \Proba_{\bsigma,\overline{\btau}}^{x,y}\left(\sqrt{N}\leq T_i = \min_{k\in\{1,\dots,l\}} T_k\right),
\end{equation*}
which is at most
\begin{equation*}
l\Proba_{\nu_\mathcal{G}}^{y}\left(\sqrt{N}\leq \min_{k\in\{1,\dots,l\}} T_k\right),
\end{equation*}
which itself is less than $\varepsilon$ for $N$ large enough.
\end{proof}

\subsection{Player 1 uniformly guarantees \texorpdfstring{$w_\infty$}{w} in \texorpdfstring{$\Gamma$}{Gamma}}
\label{subsec_P1unifvalue}

\subsubsection{The auxiliary games \texorpdfstring{$(\widetilde{\Gamma}_i)_{i\in\{1,\dots,L\}}$}{Gamma tilde i}}
Let us fix $\varepsilon>0$ and $N_0\in\N^\ast$ accordingly as in \cref{prop_unif_Gamma_i}. Beware that the objects that we now introduce also depend on $\varepsilon$.

We construct $L$ auxiliary games $(\widetilde{\Gamma}_i)_{i\in\{1,\dots,L\}}$ which are copies of the games $(\Gamma_i)_{i\in\{1,\dots,L\}}$ with an additional clock keeping track of the time modulo $N_0$. The state of player 2 on $C_i\times[1,N_0]$, where $[1,N_0] = \{1,\dots, N_0\}$, moves on $C_i$ as in $\Gamma_i$ and on $[1,N_0]$ by adding $1$ at each stage (and starting back to $1$ when $N_0+1$ is reached). For all $i\in\{1,\dots,L\}$, \[\widetilde{\Gamma}_i = (X,C_i\times[1,N_0],A,(B_y)_{y\in C_i},p,\tilde{q},u),\]
and
\[\begin{array}{ccccc}
\tilde{q} & : & C_i\times[1,N_0]\times (B_y)_{y\in C_i} & \to & \Delta(C_i\times[1,N_0]) \\
 & & (y,t,b) & \mapsto & q(\cdot|y,b)\otimes \delta_{t+1}, \\
\end{array}\]
where $t$ is taken modulo $N_0$ in $\delta_{t+1}$. 

The purpose of these games is that they have the same uniform value as the $\Gamma_i$'s, but $\varepsilon$-optimal stationary strategies instead of $\varepsilon$-optimal Markov $N_0$-periodic strategies. Indeed, by \cref{prop_unif_Gamma_i}, let $\bsigma^i$ be an $\varepsilon$-optimal Markov $N_0$-periodic strategy in $\Gamma_i$. Let $\tilde{\bmu}^{i}$ be the following stationary strategy in $\widetilde{\Gamma}_i$: for all $(x,y,t)\in X\times C_i\times [1,N_0]$,
$\tilde{\bmu}^i(x,y,t) = \bsigma^i_t(x,y)$. This defines a stationary strategy of player 1 in $\widetilde{\Gamma}_i$ which is $\varepsilon$-optimal. 

\begin{remark}
Since we are dealing with Markov periodic strategies, one could think of defining maximal communicating sets with regard to this class of strategies rather than stationary strategies as in \cref{def_stationaryrecurrentclass}. However the construction we propose here appears to provide a simpler demonstration of player 1 uniformly guaranteeing $w_\infty$ in $\Gamma$.
\end{remark}

\subsubsection{The auxiliary game \texorpdfstring{$\widetilde{\Gamma}$}{Gamma tilde}}
We now gather together the games $\widetilde{\Gamma}_i$ into one game $\widetilde{\Gamma}$.
Let $\widetilde{Y} = Y\times [1,N_0].$ The game $\widetilde{\Gamma}$ is defined by $\widetilde{\Gamma}=\left(X,\widetilde{Y},A,B,p,\tilde{q},u\right).$
Where $\tilde{q}$ is extended on $\widetilde{Y}\times B$ as follows.

\[\begin{array}{ccccc}
\tilde{q} & : & \widetilde{Y}\times B & \to & \Delta(\widetilde{Y}) \\
 & & (y,t,b) & \mapsto &  \tilde{q}(\cdot|y,t,b),\\
\end{array}\]
where $\tilde{q}\left(y',t+1|y,t,b\right)=q(y'|y,b)$ if there exists $i\in\{1,\dots,L\}$ such that $y,y'\in C_i$, and $t$ is taken modulo $N_0$.
$\tilde{q}\left(y',1|y,t,b\right)=q(y'|y,b)$ if there exists $i\in\{1,\dots,L\}$ such that $y\in C_i$ and $y'\notin C_i$, or $y\in D$. Otherwise $\tilde{q}\left(y',t'|y,t,b\right)=0$.
In words, the clock on $[1,N_0]$ is incremented as in the games $\widetilde{\Gamma}_i$ as long as the state stays in $C_i$, and is reset to $1$ when the state jumps from one $C_i$ to another or to a transient state, or from a transient state to any other state.

Let $\widetilde{\bScal}$ and $\widetilde{\bTcal}$ be the set of strategies of player 1 and 2 respectively in $\widetilde{\Gamma}$.
It is important to note that any quantity guaranteed by a player in $\widetilde{\Gamma}$ is also guaranteed in $\Gamma$.

The following lemma states that the $Y$ component of a recurrent class in $X\times\widetilde{Y}$ cannot have nonempty intersection with two different maximal communicating sets in $Y$.

\begin{lemma}
\label{lem_rec_class_tilde}
Let $(\tilde{\bmu},\tilde{\bnu})$ be a pair of stationary strategies on $X\times \widetilde{Y}$. Let $R\subseteq X\times\widetilde{Y}$ be a recurrent class under $(\tilde{\bmu},\tilde{\bnu})$. 
Then there exists $i\in\{1,\dots,L\}$ such that $R \subseteq X\times C_i \times [1,N_0].$
\end{lemma}

\begin{proof}
Let $R' = \{(x,y)\in X\times Y \ | \ \exists t\in [1,N_0] \ (x,y,t)\in R\}$ be the projection of $R$ over $X\times Y$. For all $(x,y)\in R'$ let $R_{x,y} = \{t\in [1,N_0] \ | \ (x,y,t)\in R\}$.

Define the stationary strategies $\bmu$ and $\bnu$ on $X\times Y$ by, for all $(x,y)\in R'$,
\[\bmu(x,y) = \frac{1}{|R_{x,y}|} \sum_{t\in R_{x,y}} \tilde{\bmu}(x,y,t) \text{ and } \bnu(x,y) = \frac{1}{|R_{x,y}|} \sum_{t\in R_{x,y}} \tilde{\bnu}(x,y,t),\]
and arbitrarily outside $R'$.
Under $(\bmu,\bnu)$, $R'$ is a recurrent class. Hence by \cref{prop_jointscc} there exists $i\in\{1,\dots, L\}$ such that $R'\subseteq X\times C_i.$
\end{proof}

For all $i\in \{1,\dots,L\}$, let $\tilde{\bmu}^{i}$ be an $\varepsilon$-optimal stationary strategy of player 1 in $\widetilde{\Gamma}_i$. Define the stationary strategy $\widetilde{\bmu}$ of player 1 in $\widetilde{\Gamma}$ as follows. For all $(x,y,t)\in X\times\widetilde{Y}$ \[\tilde{\bmu}(x,y,t) = \begin{cases} \tilde{\bmu}^{i}(x,y,t) \text{ if there exists } i\in\{1,\dots,L\} \text{ such that } y\in C_i\\ \text{arbitrary fixed action if } y\in D. \end{cases}\]

In the game $\widetilde{\Gamma}$, by \cref{prop_bestreply}, player 2 has a stationary strategy $\tilde{\bnu}\in\Delta(B)^{X\times \widetilde{Y}}$ and there exists $M\in \N^\ast$ such that for all $N\geq M$ and all $\btau\in\widetilde{\mathcal{T}}$ \[\gamma_N(\tilde{\bmu},\tilde{\bnu})(x,y,t)\leq\gamma_N(\tilde{\bmu},\btau)(x,y,t)+\varepsilon.\]

\sloppy The pair $(\tilde{\bmu},\tilde{\bnu})$ induces a Markov chain on $X\times\widetilde{Y}$ with recurrent classes $R_1,\dots,R_m$. By \cref{lem_rec_class_tilde} there exists a mapping $\psi : \{1,\dots,m\}\to\{1,\dots,L\}$ such that for all $i\in\{1,\dots,m\}$ \[R_i\subseteq X\times C_{\psi(i)}\times [1,N_0].\]

For all $i\in\{1,\dots,m\}$ we define $\widetilde{T}_i = \min\{n\in \N^\ast \ | \ (X_n,Y_n,t_n) \in R_i\}$ to be the hitting time of $R_i$ by $(X_n,Y_n,t_n)_{n\geq 1}$.

\begin{lemma}
\label{lem_player1guarantee}
In the game $\widetilde{\Gamma}$ starting at $(x,y,1)\in X\times \widetilde{Y}$, player 1 uniformly guarantees \[\sum_{i=1}^m v^{\psi(i)}_\infty \Proba_{\tilde{\bmu},\tilde{\bnu}}^{x,y,1}\left(\widetilde{T}_i = \min_{k\in\{1,\dots,m\}} \widetilde{T}_k\right)-3\varepsilon.\]
\end{lemma}

\begin{proof}Let $N\geq 2$ and $\varepsilon'>0$,
\begin{align*}
\frac{1}{N}\E_{\tilde{\bmu},\tilde{\bnu}}^{x,y,1}\left[\sum_{n=1}^N u_n\right]&\geq \sum_{i=1}^m \frac{1}{N}\E_{\tilde{\bmu},\tilde{\bnu}}^{x,y,1}\left[\left(\sum_{n=\widetilde{T}_i+1}^N u_n\right) \1\left(\sqrt{N}> \widetilde{T}_i = \min_{k\in\{1,\dots,m\}} \widetilde{T}_k\right) \right].
\end{align*}
This last quantity is equal to
\begin{align*}
\sum_{i=1}^m \E_{\tilde{\bmu},\tilde{\bnu}}^{x,y,1}\left[\frac{N-\widetilde{T}_i}{N}\frac{1}{N-\widetilde{T}_i}\left(\sum_{n=\widetilde{T}_i+1}^N u_n\right) \1\left(\sqrt{N}> \widetilde{T}_i = \min_{k\in\{1,\dots,m\}} \widetilde{T}_k\right) \right].
\end{align*}
Since $\widetilde{T}_i$ is taken less than $\sqrt{N}$ in the indicator function, the latter is greater than
\begin{align*}
\sum_{i=1}^m \left(1-\frac{\sqrt{N}}{N}\right) &\E_{\tilde{\bmu},\tilde{\bnu}}^{x,y,1}\left[\frac{1}{N-\widetilde{T}_i}\sum_{n=\widetilde{T}_i+1}^N u_n \Bigg\vert\sqrt{N}> \widetilde{T}_i = \min_{k\in\{1,\dots,m\}} \widetilde{T}_k\right]\\&\times\Proba_{\tilde{\bmu},\tilde{\bnu}}^{x,y,1}\left(\sqrt{N}> \widetilde{T}_i = \min_{k\in\{1,\dots,m\}} \widetilde{T}_k\right).
\end{align*}
Recall that for all $i\in\{1,\dots,L\}$, $\tilde{\bmu}^{i}$ is $\varepsilon$-optimal in $\widetilde{\Gamma}_i$, which has value $v^i_\infty$. Moreover, conditionally on $\widetilde{T}_i = \min_{k\in\{1,\dots,m\}} \widetilde{T}_k$, from $\widetilde{T}_i +1$ on, player 2 only plays actions that are in $(B_y)_{y\in R_i}$, otherwise the state process on $Y$ would have a positive probability of leaving $C_{\psi(i)}$ and $R_i$ would not be a recurrent class. Hence for $N$ large enough the quantity above is greater than or equal to
\begin{align*}
\sum_{i=1}^m \left(1-\frac{\sqrt{N}}{N}\right) \left(v^{\psi(i)}_\infty-\varepsilon\right)\Proba_{\tilde{\bmu},\tilde{\bnu}}^{x,y,1}\left(\sqrt{N}> \widetilde{T}_i = \min_{k\in\{1,\dots,m\}} \widetilde{T}_k\right),
\end{align*}
which, for $N$ large enough is greater than
\begin{align*}
\left(1-\varepsilon'\right)\sum_{i=1}^m \left(v^{\psi(i)}_\infty-\varepsilon\right)\left(\Proba_{\tilde{\bmu},\tilde{\bnu}}^{x,y,1}\left(\widetilde{T}_i = \min_{k\in\{1,\dots,m\}} \widetilde{T}_k\right)-\varepsilon'\right).
\end{align*}
Hence there exists $\widetilde{N}\in\N^\ast$ such that for all $N\geq \widetilde{N}$ and all $\btau\in\widetilde{\mathcal{T}}$,
\begin{align*}
\gamma_N(\tilde{\bmu},\btau)(x,y,1)+\varepsilon  \geq& \left(1-\varepsilon'\right)\sum_{i=1}^m v^{\psi(i)}_\infty\left(\Proba_{\tilde{\bmu},\tilde{\bnu}}^{x,y,1}\left(\widetilde{T}_i = \min_{k\in\{1,\dots,m\}} \widetilde{T}_k\right)-\varepsilon'\right)\\ &-\varepsilon(1-\varepsilon').
\end{align*}
Finally, for $\varepsilon'$ small enough, the right-hand side is greater than
\[\sum_{i=1}^m v^{\psi(i)}_\infty \Proba_{\tilde{\bmu},\tilde{\bnu}}^{x,y,1}\left(\widetilde{T}_i = \min_{k\in\{1,\dots,m\}} \widetilde{T}_k\right)-2\varepsilon,\]
which concludes the proof.
\end{proof}

To conclude, we prove that the payoff uniformly guaranteed (up to $3\varepsilon$) by player 1 in $\widetilde{\Gamma}$ in \cref{lem_player1guarantee} can be obtained by player 2 as a limit payoff in the MDP $\mathcal{G}$, and hence is greater than $w_\infty$, the uniform value of $\mathcal{G}$. 

\begin{proposition}
\label{prop_player1guarantees}
Player 1 uniformly guarantees $w_\infty$ in $\Gamma$.
\end{proposition}

\begin{proof}
By \cref{lem_player1guarantee}, in the game $\widetilde{\Gamma}$ starting at $(x,y,1)\in X\times \widetilde{Y}$ player 1 uniformly guarantees \[\sum_{i=1}^m v^{\psi(i)}_\infty\Proba_{\tilde{\bmu},\tilde{\bnu}}^{x,y,1}\left(\widetilde{T}_i = \min_{k\in\{1,\dots,m\}} \widetilde{T}_k\right)-3\varepsilon.\] 

In the MDP $\mathcal{G}$, let $\tau_\mathcal{G}$ be the following policy of player 2. At stage $1$, play $\tilde{\bnu}(x,y,1)$, a new state $(y_2,t_2)\in \widetilde{Y}$ is selected according to $\tilde{q}(\cdot |y,1,\tilde{\bnu}(x,y,1))$. At stage $2$, play $\tilde{\bnu}(x_2,y_2,t_2)$ where $x_2\in X$ is selected according to $p(\cdot|x,\tilde{\bmu}(x,y,1))$. A new state $(y_3,t_3)\in \widetilde{Y}$ is selected according to $\tilde{q}(\cdot |y_2,t_2,\tilde{\bnu}(x_2,y_2,t_2))$. Inductively at stage $n>2$, play $\tilde{\bnu}(x_n,y_n,t_n)$ where $x_n\in X$ is selected according to $p(\cdot|x_{n-1},\tilde{\bmu}(x_{n-1},y_{n-1},t_{n-1}))$.

$\tau_\mathcal{G}$ has the following interpretation. Player 2 plays in $\mathcal{G}$ according to $\tilde{\bnu}\in \Delta(B)^{X\times \widetilde{Y}}$ by simulating at each stage a fictitious state of player 1 on $X$ that follows $\tilde{\bmu}\in \Delta(B)^{X\times \widetilde{Y}}$. This induces a fictitious hitting times $\widetilde{T}_i$'s of the recurrent classes $R_1,\dots,R_m$.

Under $\tau_\mathcal{G}$ the laws of the $ \widetilde{T}_i$'s are the same as under $(\tilde{\bmu},\tilde{\bnu})$. After the process $(X_n,Y_n,t_n)$ has reached $R_i$, the payoff to player 2 in $\mathcal{G}$ is $v^i_\infty$ at each stage.

Therefore, $\tau_\mathcal{G}$ yields in $\mathcal{G}$ a limit payoff of \[\sum_{i=1}^m v^{\psi(i)}_\infty\Proba_{\tilde{\bmu},\tilde{\bnu}}^{x,y,1}\left(\widetilde{T}_i = \min_{k\in\{1,\dots,m\}} \widetilde{T}_k\right).\] Hence this quantity is greater than $w_\infty(y)$.

Since in the initial game $\Gamma$ starting at $(x,y)$, player 1 also uniformly guarantees \[\sum_{i=1}^m v^{\psi(i)}_\infty\Proba_{\tilde{\bmu},\tilde{\bnu}}^{x,y,1}\left(\widetilde{T}_i = \min_{k\in\{1,\dots,m\}} \widetilde{T}_k\right)-3\varepsilon,\] he uniformly guarantees $w_\infty(y)-3\varepsilon$.

This is true for all $\varepsilon>0$. Hence player 1 uniformly guarantees $w_\infty$ in $\Gamma$.
\end{proof}

\section{Proof of \texorpdfstring{Theorem~\ref{THM_WEAK}}{Theorem~\ref{THM_WEAK}}}
\label{sec_proofthweak}

In \cref{subsec_counterexample} we present the weakly communicating on both sides zero-sum product stochastic game used to prove \cref{THM_WEAK}. The example proposed shows that in weakly communicating on both sides games, even if players can go from any state to any other state of their component in finite time, they can make mistakes that are irreversible with regards to the joint state. \cref{subsec_shapleysimplification} is dedicated to the simplification of the Shapley equations obtained for the weakly communicating on both sides game (\cref{prop_equationsknown}). Finally in \cref{subsec_examplerenault} we show that the Shapley equations obtained after simplification are also those of a simple game of perfect information, i.e., in each state only one player controls the transition, with two absorbing and two non absorbing states. The latter game does not have an asymptotic value (\cref{lem_x-y,lem_vleq1/2,lem_v=1/2,lem_limsupleq4/9}).

\subsection{A counterexample}
\label{subsec_counterexample}
\subsubsection{State spaces and action sets}
The state space of player 1 is $X = \{x,y\}\times C_8 $ where $C_8 = \Z/8\Z$.
The state space of player 2 is $Y = \{x',y'\}\times C'_8 $ where $C'_8 =  \Z/8\Z$.

Let $I = \{0\}\cup \left\lbrace 1/2^{2n} \ \vert \ n\geq 1\right\rbrace$, and $J=[0,1/4]$. Let $A = I\times \{-1,+1\}\cup \{0,1\}\times \{0\}$ and $B = J\times \{-1,+1\}\cup \{0,1\}\times \{0\}$ be the action sets of player 1 and 2 respectively. 

It is essential that $I$, and therefore $A$, are not semi-algebraic. Indeed, since $X$ and $Y$ are finite, and the transitions we define below are polynomial, if $A$ and $B$ were definable in some o-minimal structure, the game would have a uniform value, see \citep[Theorem 4]{bolteetal2014}.

\subsubsection{Transitions}
For $i\in \{x,y\}$ we denote by $-i$ the element of $\{x,y\}\backslash \{i\}$.

In state $(i,k)\in X$ if player 1 plays $(\alpha,p)\in I\times\{-1,+1\}$ then with probability $1-\alpha-\alpha^2$ the new state is $(i,k+p)$, with probability $\alpha$ the new state is $(-i,k + p)$, and with probability $\alpha^2$ the new state is $(i,k-p)$ (see \cref{transition_1}).

Still in state $(i,k)\in X$, if player 1 plays $(\alpha,0)\in \{0,1\}\times \{0\}$, then with probability $1-\alpha$ the state remains in $(i,k)$ and with probability $\alpha$ the new state is $(-i,k)$ (see \cref{transition_0}).

Transitions for player 2 are analogous on $Y$.

\begin{figure}[ht]
\begin{center}
\begin{tikzpicture}[scale=0.6]
\node at (-5,2){$x$};
\node at (-5,-2){$y$};
\draw[dotted,->,NavyBlue,thick] (-5.5,2.5) to [out=135,in=180] (-5,4)
	to [out=0,in=45] (-4.5,2.5);
\node [below] at (-5,4) {$\alpha^2$};
\draw[dashed,->,Mahogany] (-5.5,1.5) to [out=225 ,in=135] (-5.5,-1.5);
\node [left] at (-6,0) {$\alpha$};
\draw[->,OliveGreen] (-5.5,2) to [out=180,in=180] (-5,5)
	to [out=0,in=0] (-4.5,2);
\node [above] at (-5,5) {$1-\alpha-\alpha^2$};

\draw (-5,2) circle (0.4);
\draw (-5,-2) circle (0.4);

\draw (0,0) circle [radius=2];
\node [below] at (0,-2) {$0$};
\draw (0,-1.9)--(0,-2.1);
\node [above] at (0,2) {$4$};
\draw (0,1.9)--(0,2.1);
\node [right] at (2,0) {$2$};
\draw (1.9,0)--(2.1,0);
\node [left] at (-2,0) {$6$};
\draw (-1.9,0)--(-2.1,0);

\node [below right] at (1.3,-1.3) {$1$};
\draw (1.35,-1.35)--(1.5,-1.5);
\node [above left] at (-1.3,1.3) {$5$};
\draw (1.35,1.35)--(1.5,1.5);
\node [above right] at (1.3,1.3) {$3$};
\draw (-1.35,1.35)--(-1.5,1.5);
\node [below left] at (-1.3,-1.3) {$7$};
\draw (-1.35,-1.35)--(-1.5,-1.5);

\draw[dashed,->,Mahogany] (2.6,0) to [out=0,in=270] (3,1.3) to [out=90,in=45] (2,2);
\node[below left] at(3,1.4) {$\alpha$};

\draw[dotted,->,NavyBlue,thick] (2.6,0) to [out=0,in=90] (3,-1.3) to [out=270,in=315] (2,-2);
\node[right] at(3,-1.4) {$\alpha^2$};

\draw[->,OliveGreen] (2.6,0) to [out=0,in=270] (4,1.7) to [out=90,in=90] (1.5,2);
\node[above right] at(4,1.7) {$1-\alpha-\alpha^2$};
\end{tikzpicture}
\end{center}
\caption{Transition of player 1 when playing $(\alpha,+1)$, $\alpha\in I$ in state $(x,2)$}
\label{transition_1}
\end{figure}
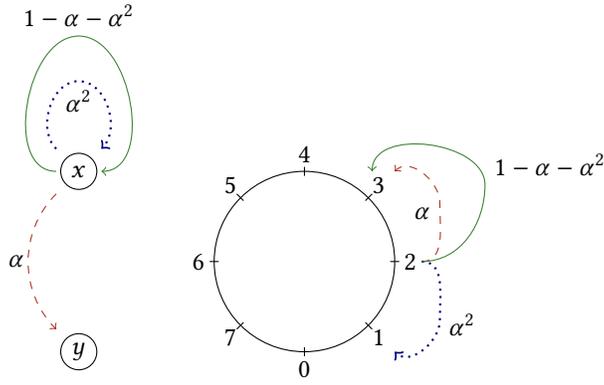

\begin{figure}[ht]
\begin{center}
\begin{tikzpicture}[scale=0.6]
\node at (-5,2){$x$};
\node at (-5,-2){$y$};
\draw[->,OliveGreen] (-5.5,2.5) to [out=135,in=180] (-5,4)
	to [out=0,in=45] (-4.5,2.5);
\node [above] at (-5,4) {$1-\alpha$};
\draw[dashed,->,Mahogany] (-5.5,1.5) to [out=225 ,in=135] (-5.5,-1.5);
\node [left] at (-6,0) {$\alpha$};

\draw (-5,2) circle (0.4);
\draw (-5,-2) circle (0.4);

\draw (0,0) circle [radius=2];
\node [below] at (0,-2) {$0$};
\draw (0,-1.9)--(0,-2.1);
\node [above] at (0,2) {$4$};
\draw (0,1.9)--(0,2.1);
\node [right] at (2,0) {$2$};
\draw (1.9,0)--(2.1,0);
\node [left] at (-2,0) {$6$};
\draw (-1.9,0)--(-2.1,0);

\node [below right] at (1.3,-1.3) {$1$};
\draw (1.35,-1.35)--(1.5,-1.5);
\node [above left] at (-1.3,1.3) {$5$};
\draw (1.35,1.35)--(1.5,1.5);
\node [above right] at (1.3,1.3) {$3$};
\draw (-1.35,1.35)--(-1.5,1.5);
\node [below left] at (-1.3,-1.3) {$7$};
\draw (-1.35,-1.35)--(-1.5,-1.5);

\draw [dashed,->,Mahogany] (2.6,-0.4) to [out=315,in=270] (3.5,0) to [out=90,in=45] (2.6,0.4);
\node [left] at (3.5,0) {$\alpha$};

\draw [->,OliveGreen] (2.4,-0.6) to [out=290,in=270] (4,0) to [out=90,in=70] (2.4,0.6);
\node [right] at (4,0) {$1-\alpha$};
\end{tikzpicture}
\end{center}
\caption{Transition of player 1 when playing $(\alpha,0)$, $\alpha\in\{0,1\}$ in state $(x,2)$}
\label{transition_0}
\end{figure}
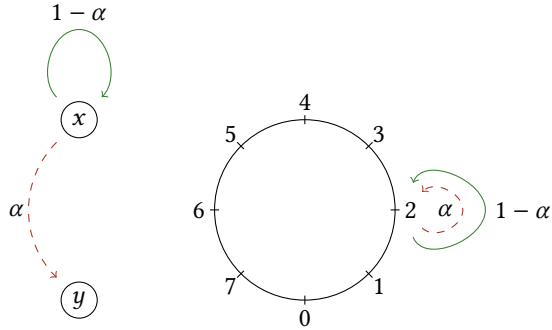

Hence by playing $(0,p)$, player 1 totally controls the dynamics on $C_8$, and by playing $(\alpha,0)$ with $\alpha$ equal to $0$ or $1$ he totally controls the dynamics on $\{x,y\}$. And likewise for player 2.

\subsubsection{Payoffs}
Let $(i,k)\in X$ and $(i',k')\in Y$. We denote by $d_{C_8}(k,k')$ the distance between player 1 and 2 on the circle $C_8$.

The payoff function $u$ is defined as follows.
If $d_{C_8}(k,k')\geq 3$ then $u((i,k),(i',k')) = 1$.
If $d_{C_8}(k,k')\leq 1$ then $u((i,k),(i',k')) = 0$.
Otherwise, if $d_{C_8}(k,k') = 2$ then $u$ is defined by the following table:
\begin{align*}
\begin{tabular}{|c|c|c|}
\hline 
 $u(\cdot,\cdot)$ & $x'$ & $y'$ \\ 
\hline 
$x$ & $0$ & $1$ \\ 
\hline 
$y$ & $1$ & $0$ \\ 
\hline 
\end{tabular}
\end{align*}

The interpretation of the game is the following. Player 1 wants to maximize his distance to player 2 who wants to minimize his distance to player 1. If the distance between them on the circle is at most $1$ or at least $3$, then their positions in $\{x,y\}$ and $\{x',y'\}$ do not matter. Whereas if the distance between them on the circle is equal to $2$, then player 1 wants to be in $x$ (resp. $y$) when player 2 is in $y'$ (resp. $x'$).

Note also that if the distance between the players is at least $3$ (resp. at most $1$), then player 1 (resp. player 2) can play such that the distance is always at least $3$ (resp. at most $1$), and this is optimal for him. Hence those joint states on $X\times Y$ act as absorbing states with payoff $1$ and $0$ respectively.

\subsubsection{Shapley equations}
For $p\in \{-1,0,+1\}$ we denote as well $p$ by the triplet $(\pj,\pjj,\pjjj)\in \{0,1\}^3$ with $\pj = 1$ if $p=-1$ and $\pj=0$ otherwise, and likewise for $\pjj$ and $\pjjj$. We define similarly $q = (\qj,\qjj,\qjjj)$.

Let $(i,k)\in X$ and $(i',k')\in Y$ be the initial states of player 1 and 2.
By symmetry of the game, we only consider the cases $(i,k)\in\{(x,0),(x,1),(x,2),(x,3),(x,4)\}$ and $(i',k')\in\{(x',0),(y',0)\}$.
Moreover, the joint state $((i,k),(i',0))$ is denoted $(i,i',k)$.

Let $\lambda\in (0,1)$. Let $\x = \vl(x,y',2)$ and $\y=\vl(x,x',2)$.
Clearly $\vl(\cdot,\cdot,3)=\vl(\cdot,\cdot,4)=1$ and $\vl(\cdot,\cdot,0)=\vl(\cdot,\cdot,1)=0$. $\x$ is the value of the game, played with mixed strategies, where player 1 chooses $(\alpha,p)\in A$ and player 2 chooses $(\beta,q)\in B$ with payoff $\lambda + (1-\lambda)h(\x,\y,\alpha,p,\beta,q)$ where $h(\x,\y,\alpha,p,\beta,q)$ equals
\begin{align*}
\begin{split}
&\bigg[ 
	\left(
		(1-\alpha-\alpha^2)(1-\beta-\beta^2) + \alpha\beta+\alpha^2\beta^2
		\right)(\pj\qj+\pjjj\qjjj)\\
&+
	\left(
		(1-\alpha)(1-\beta)+\alpha\beta
	\right)\pjj\qjj\\
&+
	\left(
		(1-\alpha-\alpha^2)\beta^2 + (1-\beta-\beta^2)\alpha^2
	\right)(\pj\qjjj+\pjjj\qj)
\bigg]\x\\
&+
\bigg[
	\left(
		(1-\alpha-\alpha^2)\beta + (1-\beta-\beta^2)\alpha
	\right)(\pj\qj+\pjjj\qjjj)\\
	&+\left((1-\alpha)\beta+(1-\beta)\alpha\right)\pjj\qjj + (\alpha\beta^2+\beta\alpha^2)(\pj\qjjj+\pjjj\qj)
\bigg]\y\\
&+\alpha^2(1-\beta^2)\pj\qj+\beta^2(1-\alpha^2)\pjjj\qjjj+(1-\alpha^2)\pjjj\qjj+\alpha^2\pj\qjj\\
&+\beta^2\pjj\qjjj+(1-\beta^2)\pjj\qj+(1-\alpha^2)(1-\beta^2)\pjjj\qj+\alpha^2\beta^2\pj\qjjj.
\end{split}
\end{align*}

$\y$ is the value of the game, played with mixed strategies, where player 1 chooses $(\alpha,p)\in A$ and player 2 chooses $(\beta,q)\in B$ with payoff $(1-\lambda)h(\y,\x,\alpha,p,\beta,q)$.

\subsection{Simplification of \texorpdfstring{$(\x)_{\lambda\in(0,1)}$}{xlambda} and \texorpdfstring{$(\y)_{\lambda\in(0,1)}$}{ylambda}}
\label{subsec_shapleysimplification}

The aim of this section is to simplify the expressions of $\x$ and $\y$ obtained in the previous section.

Let $\lambda\in (0,1)$. It is clear that $\x>0$, therefore $\y>0$ because player 1 can play $(\alpha,+1)$ with $\alpha>0$ in $(x,x',2)$. Moreover player 2 can play $(\beta,+1)$ with $\beta>0$ in $(x,y',2)$ and it is easy to check that $\x<1$.

\begin{lemma}
\label{lem_dominant}
For all $\lambda\in (0,1)$ the following equations hold.
\begin{align}
\label{eq1}
&\x>\y\\
\label{eq2}
\begin{split}
&\y =(1-\lambda)\max_{\substack{(\alpha,p)\in A}} \left(\alpha^2(\pj-\pjjj)\y+\alpha\pjjj(\x-\y)+\pjjj\y\right)
\end{split} \\ 
\label{eq3}
&\x = \lambda + (1-\lambda)\min_{\substack{(\beta,q)\in B}} \bigg(\beta^2(\qj-\qjjj)(\x-1)+\beta\qjjj(\y-\x)+\qjjj(\x-1)+1\bigg).
\end{align}
\end{lemma}
The proof of \cref{lem_dominant} is postponed to \cref{se:appendix}. In the next proposition, we show that $\x$ and $\y$ indeed have a rather simple expression.
\begin{proposition}
\label{prop_equationsknown}
For all $\lambda\in (0,1)$,
\begin{align}
\label{eq4}
&\lambda\y =(1-\lambda)\max_{\alpha\in I} \left(-\alpha^2 \y + \alpha(\x-\y)\right)\\
\label{eq5}
&\lambda\x = \lambda + (1-\lambda)\min_{\beta\in J} \left(\beta^2(1-\x)+\beta(\y-\x)\right).
\end{align}
\end{proposition}
The proof of \cref{prop_equationsknown} is postponed to \cref{se:appendix}.

\subsection{A simple counterexample to the convergence of \texorpdfstring{$(\x)_{\lambda\in(0,1)}$}{xlambda} and \texorpdfstring{$(\y)_{\lambda\in(0,1)}$}{ylambda}}
\label{subsec_examplerenault}

\cref{eq4,eq5} are in fact the Shapley equations of the following game, non product, non weakly communicating on both sides, described in \cref{fig_knowngame}. This game does not have a product state space, neither has it a weakly communicating property. This example was introduced by \cite{renault:hal-02130451}. The state space is $\Omega = \{0,1,0^\ast,1^\ast\}$. The action space of player 1 (resp. 2) is $I$ (resp. $J$). Player 1 (resp. 2) plays in state $0$ (resp. $1$). The states $0^\ast$ and $1^\ast$ are absorbing. The payoff in states $0$ and $0^\ast$ (resp. $1$ and $1^\ast$) is $0$ (resp. $1$).
\begin{figure}[ht]
\begin{center}
\begin{tikzpicture}[scale=0.6]
\node at (-2,2){$0$};
\node at (2,2){$1$};
\node at (-2,-2) {$0^\ast$};
\node at (2,-2) {$1^\ast$};
\draw (-2,2) circle (0.4);
\draw (2,2) circle (0.4);
\draw (-2,-2) circle (0.4);
\draw (2,-2) circle (0.4);
\draw[->] (-2,1.4) to (-2,-1.4);
\node [left] at (-2,0) {$\alpha^2$};
\draw[->] (2,1.4) to (2,-1.4);
\node [right] at (2,0) {$\beta^2$};
\draw[->] (-1.4,2.4) to [out=45 ,in=135] (1.4,2.4);
\node [above] at (0,3) {$\alpha$};
\draw[->] (1.4,1.6) to [out=-135 ,in=-45] (-1.4,1.6);
\node [below] at (0,1) {$\beta$};
\draw [->] (2.6,1.6) to [out=315,in=270] (3.5,2) to [out=90,in=45] (2.6,2.4);
\node [right] at (3.5,2) {$1-\beta-\beta^2$};
\draw [->] (-2.6,1.6) to [out=-135,in=270] (-3.5,2) to [out=90,in=135] (-2.6,2.4);
\node [left] at (-3.5,2) {$1-\alpha-\alpha^2$};
\end{tikzpicture}
\end{center}
\caption{A simpler game}
\label{fig_knowngame}
\end{figure}
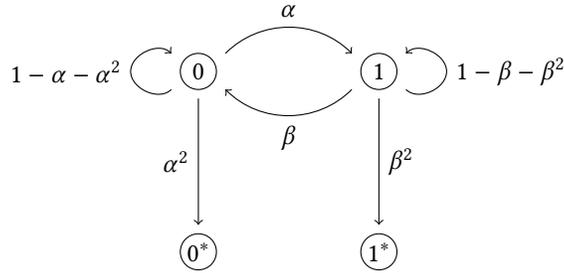

We conclude with the next lemmas showing that $(\x)_{\lambda\in(0,1]}$ and $(\y)_{\lambda\in(0,1]}$ do not converge. The proofs are provided for completeness.

\begin{lemma}
\label{lem_x-y}
For all $\lambda\leq 1/17$, $\beta_\lambda = \frac{\x-\y}{2(1-\x)}$ is optimal for player 2. Moreover
\begin{equation}
\label{eq_1-x}
4\lambda (1-\x)^2 = (1-\lambda)(\x-\y)^2.
\end{equation}
Hence $\x-\y$ goes to $0$ as $\lambda$ goes to $0$.
\end{lemma}

\begin{proof}
Suppose that $\beta_\lambda\leq 1/4$, then it is easy to verify that the lemma holds.
Suppose now that $\beta_\lambda > 1/4$. Then the minimum in \cref{eq5} is attained in $1/4$. It yields
\[\frac{1}{4}(1-\lambda)(\x-\y) = (1-\x)\left(\frac{1+15\lambda}{16}\right)>(1-\lambda)\frac{1-\x}{8}.\]
Hence $\lambda>1/17.$
\end{proof}

Let $(\lambda_n)_{n\in\N^\ast}\in(0,1]^{\N^\ast}$ such that $\lambda_n\to 0$ as $n$ goes to $+\infty$.

\begin{lemma}
\label{lem_vleq1/2}
If  $(x_{\lambda_n})_n$ and $(y_{\lambda_n})_n$ converge to $v\in[0,1]$, then $v\leq 1/2$. Moreover $x_{\lambda_n}-y_{\lambda_n}\sim 2 \sqrt{\lambda_n}(1-v)$ and $\beta_{\lambda_n}\sim \sqrt{\lambda_n}$ as $n$ goes to $+\infty$.
\end{lemma}

\begin{proof}
Let $\alpha_\lambda$ be an optimal strategy of player 1. By \cref{eq_1-x} one gets
\[2\alpha_\lambda\sqrt{\lambda}\y \leq \y(1+\alpha_\lambda^2)=\lambda\y\alpha_\lambda^2+2\alpha_\lambda\sqrt{\lambda}\sqrt{1-\lambda}(1-\x).\]
Dividing by $\alpha_\lambda\sqrt{\lambda}$ and passing to the limit yields $v\leq 1/2$.
\end{proof}

In the next two lemmas, we show that the game does not have an asymptotic value. The idea is the following: player 1 would like to play in the $\lambda$-discounted game some $\alpha$ close to $\sqrt{\lambda}\frac{1-v}{v}$ in state $0$, where $v$ is the limit (up to some subsequence) of $(x_{\lambda_n})_n$ and $(y_{\lambda_n})_n$. If player 1 is not allowed to take any absorbing risk in $[0,1/4]$, but player 2 is, we expect the values $(x_{\lambda_n})_n$ and $(y_{\lambda_n})_n$ to oscillate.

\begin{lemma}
\label{lem_v=1/2}
If for all $n\in\N^\ast$, $\sqrt{\lambda_n}\in I$ then \[\lim_{n\to+\infty} x_{\lambda_n}=1/2.\]
\end{lemma}

\begin{proof}
Assume that, up to some subsequence, $(x_{\lambda_n})_n$ and $(y_{\lambda_n})_n$ converge to some $v\in [0,1]$.
If player 1 plays $\alpha = \sqrt{\lambda}$, \cref{eq4} yields
\[\lambda\y \geq (1-\lambda)\sqrt{\lambda}(\x-\y)-(1-\lambda)\lambda\y.\]
Dividing by $\lambda$ and passing to the limit, one gets $v\geq 1/2$. By \cref{lem_vleq1/2}, $v=1/2$.
\end{proof}

\begin{lemma}
\label{lem_limsupleq4/9}
If for all $n\in\N^\ast$, $\left(1/2\sqrt{\lambda_n},2\sqrt{\lambda_n}\right)\cap I = \emptyset$ then $\limsup_{n\to+\infty} x_{\lambda_n}\leq 4/9.$
\end{lemma}

\begin{proof}
Suppose that up to some subsequence, $(x_{\lambda_n})_n$ and $(y_{\lambda_n})_n$ converge to some $v\geq 4/9$.
By \cref{lem_vleq1/2} $v\leq 1/2$.  
Let $\alpha_\lambda^\ast = \frac{\x-\y}{2\y}>0$ be the argument of the maximum of the unconstrained problem associated to \cref{eq4}. Then $\alpha_\lambda^\ast \sim \sqrt{\lambda}\frac{1-v}{v}$. Hence for $\lambda$ small enough in the sequence, $1/2\sqrt{\lambda}\leq \alpha_\lambda^\ast\leq 2 \sqrt{\lambda}.$

The open interval $\left(1/2\sqrt{\lambda_n},2\sqrt{\lambda_n}\right)$ does not contain any point in $I$. Furthermore the objective function of player 1 is increasing between $0$ and $\alpha_\lambda^\ast$, and decreasing after. 

First case, $\alpha_\lambda\leq 1/2\sqrt{\lambda}$. Then \[\lambda\y \leq 1/2(1-\lambda)\sqrt{\lambda}(\x-\y) - 1/4(1-\lambda)\lambda\y.\]
Dividing by $\lambda$ and passing to the limit yields $v\leq 4/9$.

Second case, $\alpha_\lambda\geq 2\sqrt{\lambda}$. Then \[\lambda\y \leq 2(1-\lambda)\sqrt{\lambda}(\x-\y) - 4(1-\lambda)\lambda\y,\] and again $v\leq 4/9$.
\end{proof}

Thus taking the sequences $\lambda_n = \frac{1}{2^{2n}}$ and $\lambda'_n = \frac{1}{2^{2n+1}}$, for all $n\geq1$, one has from \cref{lem_v=1/2,lem_limsupleq4/9} that $(\x)_{\lambda\in(0,1]}$ and $(\y)_{\lambda\in(0,1]}$ do not converge as $\lambda$ goes to $0$. 

\appendix
\section{Omitted proofs}
\label{se:appendix}
%\sectionmark{Appendix of Chapter 3}
% of \cref{sec_proofthweak}

\begin{proof}[Proof of \cref{lem_dominant}]
We consider the game starting in state $(x,x',2)$. Player 2 can play $(\beta,q)=(0,+1)$, hence
\begin{equation*}
\y \leq(1-\lambda)\max_{(\alpha,p)\in A} \left(\alpha^2(\pj-\pjjj)\y+\alpha\pjjj(\x-\y)+\pjjj\y\right).
\end{equation*}
Suppose \[\{0\}\in\argmax_{p} \left(\max_{\alpha\in \{0,1\}} \left(\alpha^2(\pj-\pjjj)\y+\alpha\pjjj(\x-\y)+\pjjj\y\right)\right).\]
Then $\y \leq (1-\lambda)\max_{\alpha\in \{0,1\}} 0$, hence $\y = 0$, contradiction.\\
Suppose \[\{-1\}\in\argmax_{p} \left(\max_{\alpha\in I} \left(\alpha^2(\pj-\pjjj)\y+\alpha\pjjj(\x-\y)+\pjjj\y\right)\right).\]
Then $\y\leq (1-\lambda)\max_{\alpha\in I}\alpha^2\y = \frac{1-\lambda}{16}\y$. Thus $\y=0$, contradiction.\\
Thus \[\argmax_{p} \left(\max_{\alpha} \left(\alpha^2(\pj-\pjjj)\y+\alpha\pjjj(\x-\y)+\pjjj\y\right)\right)=\{+1\}.\]
Hence $\lambda\y\leq(1-\lambda)\max_{\alpha\in I}\left(-\alpha^2+\alpha(\x-\y)\right).$
Thus $\x>\y$, and \cref{eq1} is proved.

To prove \cref{eq2}, we show that $(0,+1)$ is a dominant strategy of player 2, i.e.,
\begin{align*}
\forall (\alpha,p)\in A \ \forall (\beta,q)\in B \ h(\y,\x,\alpha,p,\beta,q)-h(\y,\x,\alpha,p,0,+1)\geq 0.
\end{align*}
There are 9 cases to test, corresponding to the different values of $(p,q)\in\{-1,0,+1\}^2$.
\paragraph{\nth{1} case: $\pjjj=\qjjj=1$}
\begin{align*}
&\left((1-\alpha-\alpha^2)(-\beta-\beta^2)+\alpha\beta+\alpha^2\beta^2\right)\y +\\ &\left((1-\alpha-\alpha^2)\beta+(-\beta-\beta^2)\alpha\right)\x + \beta^2(1-\alpha^2)\\
=&\left((1-2\alpha-\alpha^2-\alpha\beta)(\x-\y)+\beta(1-\alpha^2)(1-\y)+\beta\alpha^2\y\right)\beta \geq 0.
\end{align*}
\paragraph{\nth{2} case: $\pjjj=\qjj=1$}
$1-\alpha^2\geq 0$
\paragraph{\nth{3} case: $\pjjj=\qj=1$}
\begin{align*}
&\left((1-\alpha-\alpha^2)\beta^2+(1-\beta-\beta^2)\alpha^2 -(1-\alpha-\alpha^2)\right)\y + \\ & (\alpha\beta^2+\beta\alpha^2-\alpha)\x + (1-\alpha^2)(1-\beta^2)\\
= & \left((\beta^2-1)(1-\alpha-2\alpha^2)-\beta\alpha^2\right)\y + \\ &\left(\beta\alpha^2+\alpha(\beta^2-1)\right)\x+(1-\alpha^2)(1-\beta^2)\\
= & \beta\alpha^2(\x-\y)+\alpha(1-\beta^2)(1-\x)+ \\ &(1-\beta^2)(1-\alpha-\alpha^2)(1-\y)+(1-\beta^2)\alpha^2\y\geq 0
\end{align*}
\paragraph{\nth{4} case: $\pjj=\qjjj=1$}
$\beta^2\geq 0$
\paragraph{\nth{5} case: $\pjj=\qjj=1$}
\begin{align*}
\left((1-\alpha)(1-\beta)+\alpha\beta\right)\y+\left((1-\alpha)\beta+(1-\beta)\alpha\right)\x\geq 0
\end{align*}
\paragraph{\nth{6} case: $\pjj=\qj=1$}
$1-\beta^2\geq 0$
\paragraph{\nth{7} case: $\pj=\qjjj=1$}
\begin{align*}
&\left((1-\alpha-\alpha^2)\beta^2+(1-\beta-\beta^2)\alpha^2-\alpha^2\right)\y+(\alpha^2\beta+\alpha\beta^2)\x+\alpha^2\beta^2\\
= & \left((1-\alpha-\alpha^2)\beta^2+(-\beta-\beta^2)\alpha^2\right)\y+(\alpha^2\beta+\alpha\beta^2)\x+\alpha^2\beta^2\\
= & (1-2\alpha^2)\beta^2\y+(\alpha^2\beta+\alpha\beta^2)(\x-\y)+\alpha^2\beta^2\geq 0
\end{align*}
\paragraph{\nth{8} case: $\pj=\qjj=1$}
$\alpha^2\geq 0$
\paragraph{\nth{9} case: $\pj=\qj=1$}
\begin{align*}
&\left((1-\alpha-\alpha^2)(1-\beta-\beta^2)+\alpha\beta+\alpha^2\beta^2-\alpha^2\right)\y+\\ &\left((1-\alpha-\alpha^2)\beta+(1-\beta-\beta^2)\alpha\right)\x+\alpha^2(1-\beta^2)\\
= & \left((1-\alpha-\alpha^2)(1-\beta-\beta^2)+\alpha\beta\right)\y+\\ &\left((1-\alpha-\alpha^2)\beta+(1-\beta-\beta^2)\alpha\right)\x+\alpha^2(1-\beta^2)(1-\y)\geq 0
\end{align*}
\cref{eq2} is thus proved.
 
We now consider the game starting in state $(x,y',2)$.
To prove \cref{eq3}, we show that $(0,+1)$ is a dominant strategy of player 1, i.e.,
\begin{align*}
\forall (\alpha,p)\in A \ \forall (\beta,q)\in B \ h(\x,\y,\alpha,p,\beta,q)-h(\x,\y,0,+1,\beta,q)\leq 0.
\end{align*}
There are again 9 cases to test, corresponding to the different values of $(p,q)\in\{-1,0,+1\}^2$.

\paragraph{\nth{1} case: $\pjjj=\qjjj=1$}
\begin{align*}
&\left((1-\alpha-\alpha^2)(1-\beta-\beta^2)+\alpha\beta+\alpha^2\beta^2-(1-\beta-\beta^2)\right)\x + \\ &\left((1-\alpha-\alpha^2)\beta+(1-\beta-\beta^2)\alpha-\beta\right)\y + \beta^2(1-\alpha^2)-\beta^2\\
= &\left((-\alpha-\alpha^2)(1-\beta-\beta^2)+\alpha\beta+\alpha^2\beta^2\right)\x +\\ & \left((-\alpha-\alpha^2)\beta+(1-\beta-\beta^2)\alpha\right)\y -\alpha^2\beta^2\\
=&\alpha(1-2\beta-\beta^2-\alpha\beta)(\y-\x)-\alpha^2(1-\beta^2)\x+\alpha^2\beta^2(\x-1) \leq 0.
\end{align*}
\paragraph{\nth{2} case: $\pjjj=\qjj=1$}
$1-1\leq 0$
\paragraph{\nth{3} case: $\pjjj=\qj=1$}
\begin{align*}
&\left((1-\alpha-\alpha^2)\beta^2+(1-\beta-\beta^2)\alpha^2 -\beta^2\right)\x +\\& (\alpha\beta^2+\beta\alpha^2)\y + (1-\alpha^2)(1-\beta^2) - (1-\beta^2)\\
= & \left(\beta^2(-\alpha-\alpha^2)+\alpha^2(1-\beta-\beta^2)\right)\x+\left(\beta\alpha^2+\alpha\beta^2\right)\y-\alpha^2(1-\beta^2)\\
= & \left(\beta\alpha^2+\alpha\beta^2\right)(\y-\x)+\alpha^2(1-\beta^2)(\x-1)-\alpha^2\beta^2\x\leq 0\\
\end{align*}
\paragraph{\nth{4} case: $\pjj=\qjjj=1$}
$-(1-\beta-\beta^2)\x-\beta\y+\beta^2-\beta^2\leq 0$
\paragraph{\nth{5} case: $\pjj=\qjj=1$}
\begin{align*}
&\left((1-\alpha)(1-\beta)+\alpha\beta\right)\x+\left((1-\alpha)\beta+(1-\beta)\alpha\right)\y-1\\
= & \left((1-\alpha)\beta+\alpha(1-\beta)\right)(\y-\x)+\x-1\leq 0 
\end{align*}
\paragraph{\nth{6} case: $\pjj=\qj=1$}
$1-\beta^2-(1-\beta^2)-\beta^2\x\leq 0$
\paragraph{\nth{7} case: $\pj=\qjjj=1$}
\begin{align*}
&\left((1-\alpha-\alpha^2)\beta^2+(1-\beta-\beta^2)\alpha^2-(1-\beta-\beta^2)\right)\x+\\&(\alpha^2\beta+\alpha\beta^2-\beta)\y+\alpha^2\beta^2-\beta^2\\
= & (\alpha^2-1)\beta\y + \alpha\beta^2(\y-\x)+\\&(1-\alpha^2)(2\beta^2-1)\x+(\alpha^2-1)\beta(1-\x)\leq 0
\end{align*}
\paragraph{\nth{8} case: $\pj=\qjj=1$}
$\alpha^2-1\leq 0$
\paragraph{\nth{9} case: $\pj=\qj=1$}
\begin{align*}
&\left((1-\alpha-\alpha^2)(1-\beta-\beta^2)+\alpha\beta+\alpha^2\beta^2-\beta^2\right)\x+\\&\left((1-\alpha-\alpha^2)\beta+(1-\beta-\beta^2)\alpha\right)\y+\alpha^2(1-\beta^2)-(1-\beta^2)\\
= & \left((1-\alpha-\alpha^2)\beta+(1-\beta-\beta^2)\alpha\right)(\y-\x)+\\&\beta^2(\alpha^2-1)\x+(1-\alpha^2)(1-\beta^2)(\x-1)\leq 0
\end{align*}
Thus \cref{eq3} is proved.
\end{proof}

\begin{proof}[Proof of \cref{prop_equationsknown}]
It has been seen in the proof of \cref{lem_dominant} that 
\begin{equation*}
\argmax_{p} \left(\max_{\alpha} \left(\alpha^2(\pj-\pjjj)\y+\alpha\pjjj(\x-\y)+\pjjj\y\right)\right)=\{+1\}.
\end{equation*}
Thus one deduces \cref{eq4}.
Likewise,
\begin{equation*}
\argmin_{q} \left(\min_{\beta}\left(\beta^2(\qj-\qjjj)(\x-1)+\beta\qjjj(\y-\x)+\pjjj\y\right)\right) = \{+1\}.
\end{equation*}
And \cref{eq5} is proved.
\end{proof}

\section*{Acknowledgments}
The author wishes to express his gratitude to his Ph.D. advisor Jérôme Renault, as well as Fabien Gensbittel, Xavier Venel and Bruno Ziliotto for meaningful discussions.

The author also thanks an anonymous referee for his detailed and thorough comments, which contributed to significantly improve the article.

\bibliographystyle{apalike}
\bibliography{bib_total}

\end{document}